\documentclass[ amsfonts]{article}

\usepackage{graphics,amsmath,amsfonts,amssymb,graphicx,url,color,amsthm,cite,array,mathtools,booktabs}

\usepackage[margin=10pt,font=small,labelfont=bf,
labelsep=period]{caption}

\usepackage{hyperref}

\usepackage[margin=1.5in]{geometry}

\newcommand{\s}{\small}
\usepackage{tikz-cd,import,pdfpages,eqnarray}

\newlength{\notewidth}
\usepackage{varwidth}

\usepackage[shortlabels ]{enumitem}
\usepackage[all]{xy}
\usepackage[nottoc]{tocbibind}
\newtheorem{lemma}{Lemma}[section]
\newtheorem{proposition}[lemma]{Proposition}
\newtheorem{theorem}{Theorem}

\newtheorem{corollary}[lemma]{Corollary}

\theoremstyle{definition}
\newtheorem{remark}[lemma]{Remark}

\renewcommand{\r}{{\mathbf{r}}}
\renewcommand{\v}{{\mathbf{v}}}

\newcommand{\z}{{\mathbf{z}}}

\newcommand{\q}{{\mathbf{q}}}
\newcommand{\p}{{\mathbf{p}}}

\newcommand{\CO}{\mathrm{CO}}

\newcommand{\COto}{\CO_{2,1}}

\newcommand{\diag}{\mathrm{diag}}
\newcommand{\SL}{\mathrm{SL}}
\newcommand{\PSLt}{\mathrm{PSL}_2(\R)}
\newcommand{\GL}{\mathrm{GL}}

\newcommand{\SLt}{{\SL_2(\R)}} 
 
\newcommand{\PGL}{\mathrm{PGL}}
 
\newcommand{\SO}{\mathrm{SO}}

\newcommand{\tr}{{\rm tr}}

\newcommand{\cH}{\mathcal H}

\newcommand{\cP}{\mathcal P}
\newcommand{\ccP}{\overline\cP}

\newcommand{\cC}{\mathcal C}
\newcommand{\ccC}{\overline{\cC}}

\renewcommand{\b}{{\bf b}}
\newcommand{\R}{\mathbb{R}}
\newcommand{\C}{\mathbb{C}}

\newcommand{\Rto}{\R^{2,1}}
\newcommand{\Rot}{\R^{2,1}}
\newcommand{\RP}{\R P}
\newcommand{\RPt}{\RP^3}

\renewcommand{\qed}{\hfill$\square$}

\newcommand{\st}{\, | \,}

\newcommand{\n}{\noindent}

\newcommand{\sn}{\smallskip\n} 
\newcommand{\mn}{\medskip\noindent}

\newcommand{\be}{\begin{equation}}
\newcommand{\ee}{\end{equation}}

\newcommand{\PSL}{\mathrm{ PSL}}
\newcommand{\SU}{\mathrm{SU}}
\newcommand{\OS}{\mathcal{G}}
\newcommand{\tOS}{\widetilde{\OS}}
\newcommand{\ssec}{\subsection}

 \newcommand{\benum}{\begin{enumerate}[label=$(\mathrm{\alph*})$, left=-5px, itemsep=-2px]}

\newcounter{ticount} 
\newcommand{\ti}{\refstepcounter{ticount}\theticount}

\title {Revisiting Kepler: new symmetries of an old problem} 
\author{Gil Bor\footnote{
CIMAT, A.P. 402, Guanajuato, Gto. 36000, Mexico; 
gil@cimat.mx
}
\and
Connor Jackman\footnote{
CIMAT, A.P. 402, Guanajuato, Gto. 36000, Mexico; 
connor.jackman@cimat.mx
}
}

\date{\today}
\begin{document}

\maketitle

\begin{abstract}
%\note{modified abstract}
The  {\em Kepler orbits}  form  a 3-parameter family of  {\em unparametrized}  plane curves, consisting of  all conics sharing a focus  at a fixed point.  We study the geometry and symmetry properties of this family, as well  as natural 2-parameter subfamilies, such as those of fixed energy   or angular momentum. 

Our main result  is  that Kepler orbits  is a `flat' family, that is,  the local diffeomorphisms of the plane preserving   this  family form  a 7-dimensional local group, the maximum dimension  possible for the symmetry group of a 3-parameter family of plane curves. These symmetries are different from the well-studied `hidden' symmetries  of the Kepler problem, acting  on energy levels in the 4-dimensional phase space of the Kepler system.

Each 2-parameter subfamily of Kepler orbits with fixed non-zero energy (Kepler ellipses or hyperbolas with fixed length of major axis) admits  $\PSL_2(\R)$   as its (local) symmetry group, corresponding to one of the items of  a classification due to  A. Tresse (1896)  of 2-parameter families of plane curves admitting a 3-dimensional local group of  symmetries. The 2-parameter subfamilies with zero energy (Kepler parabolas) or fixed non-zero angular momentum are flat (locally diffeomorphic to the family of straight lines). 

These results can be proved using techniques developed in the 19th century by S. Lie to determine `infinitesimal point symmetries' of ODEs, but  our proofs  are much simpler, using a projective geometric model for the Kepler orbits (plane sections of a cone in projective 3-space). In this projective model all symmetry groups act globally. 

Another advantage of the projective model is  a duality between  Kepler's plane and Minkowski's 3-space parametrizing the space of Kepler orbits. We use this duality to deduce several results on the Kepler system, old and new. 

\end{abstract}

\tableofcontents

\section{Introduction and statement of main results}
A {\em Kepler orbit} is  a plane conic -- ellipse, parabola or hyperbola -- 
 with a focus at the origin (in case of a hyperbola  only the branch bending around the origin is taken). 
 Kepler orbits  form   a 3-parameter family of   plane curves, traced by the motions  of a  point mass subject to Newton's inverse square law:
 the radial attractive force is proportional to the inverse square of the distance to the origin. 
 We exclude `collision orbits' (lines through the origin). 
 See Figure  \ref{fig:kepler_orbits}. 

\begin{figure}[!htb]
        \centering
 \def\svgwidth{.9\textwidth}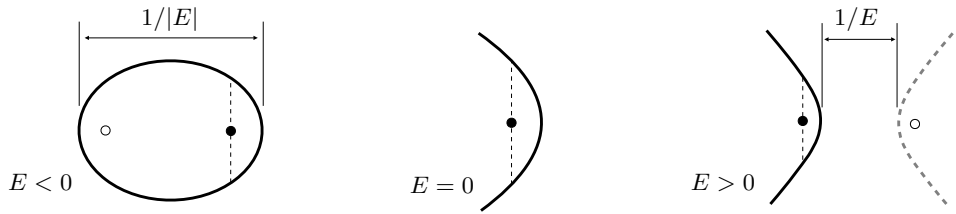
        \caption{Kepler orbit types (ellipse, parabola or hyperbola), shapes and sizes are given by  their energy $E$ and angular momentum $M$.  The major axis is $1/|E|$ and the {\em Latus rectum} (vertical dotted segment) is $2M^2$. %Eccentricity is $\sqrt{1+2EM^2}$. 
        See Section \ref{sec:reminder}. }
        \label{fig:kepler_orbits}
  \end{figure}

\ssec{Orbital symmetries} These  are local diffeomorphisms of $\R^2\setminus 0,$ taking (unparametrized) Kepler orbits   to Kepler orbits.   At the outset, it is not clear that there  are any such symmetries, local or global, other than the obvious  ones --  dilations and rotations about the origin, or reflections about lines through the origin (a 2-dimensional group of symmetries). 
Nevertheless, as we find out, there are many  additional  orbital symmetries, both for the full 3-parameter family  of Kepler orbits, as well as for some natural 2-parameter subfamilies. 
\begin{theorem}\label{thm:main1}The orbital symmetries of the Kepler problem form a 7-dimensional  group  of local diffeomorphisms of $\R^2\setminus 0$ (aka a `pseudo-group'),
%
%\note{GB: added a `pseudo-group'}
%
  the maximum dimension possible for a 3-parameter family of plane curves, generated by the following infinitesimal symmetries (vector fields whose flows act by orbital symmetries):
\be\label{eq:vf}
r\partial_r,\ 						
\partial_\theta,\ 						
r\partial_x,\  r\partial_y, \ -x r\partial_r,\ -y r\partial_r,\ 
-r^2\partial_r
\ee
\noindent (using   both Cartesian and polar coordinates). 
\end{theorem}
Note that the first two vector fields generate dilations and rotations, the `obvious' symmetries mentioned above.  {\em How about the rest of the symmetries? Where do they come from?}

%\note{added this parapgraph}
We emphasis that the 7 vector fields of Theorem \ref{thm:main1} do {\em not}  generate a honest 7-dimensional Lie group action on $\R^2\setminus 0$.   The first 4 vector fields do generate an action of the connected component of the group $\COto$ on $\R^2\setminus 0$, but the last three vector fields are in fact imcomplete (their integral curves ``run to infinity'' in finite time). As we explain later, to obtain a global group action, one needs to embedd  the Kepler plane in a larger surface, a cone in $\RP^3$, to which the above 7 vector fields extend, generating an action of the 7-dimensional subgroup of $\PGL_4(\R)$ preserving this cone. 

Now quite generally, there is a standard method for finding infinitesimal symmetries
of $n$-parameter families of plane curves, going back to S. Lie in  the 19th century, consisting of first writing down an $n$-th  order scalar ODE whose graphs of solutions form the  curves of the family. Then, one writes down  a system of PDEs for the infinitesimal symmetries of this ODE, which with some luck and skill, one can solve  explicitly. See Chapter 6 of P.\,Olver's book \cite{O}. 
This is  a straightforward albeit   tedious procedure (best left   nowadays to computers), producing the  infinitesimal symmetries  above,  but the result  remains mysterious.

Instead, our proof of Theorem \ref{thm:main1} exploits the peculiar geometry of Kepler's problem, in particular, its {\em projective} geometry, borrowing from  Lie's theory only  the upper bound  of 7 on the dimension of the symmetry group.  This proof, 
%\note{added ``rather then\ldots,''}
rather then the actual statement of  Theorem \ref{thm:main1}, is the  main thrust of this article.  See subsection \ref{sec:sketch} below for a sketch of the proof.

\ssec{The space of Kepler orbits}\label{sec:orbitspace} 

%\note{added the `Kepler plane'}
As is well known, every    Kepler orbit is the orthogonal projection onto the $xy$ plane (the `Kepler plane') of  a {\em conic section}, the intersection of the  cone $\cC:=\{x^2+y^2=z^2\}\subset \R^3$ with a plane $ax+by+cz=1,\ c\neq  0$. See Section \ref{sec:reminder} below  for a  proof (due to Lagrange \cite{L}) as well as  a reminder of some other standard facts about the Kepler problem. 
Let $\Rto$ be the 3-dimensional space  with coordinates  $(a,b,c)$  equipped with 
Minkowski's  quadratic  form $\|(a,b,c)\|^2:=a^2+b^2-c^2$ (we use this notation even though the expression has negative values!). Note that the planes $ax+by+cz=1$ and $ax+by-cz=1$ (the reflection of the former about the $xy$ plane) generate the same Kepler orbit. Thus $\Rot_+=\{c>0\}\subset \Rot$ is identified with  the space of  Kepler orbits. Furthermore, 
the cone $\|(a,b,c)\|^2=0$ parametrizes Kepler parabolas, its interior $\|(a,b,c)\|^2<0$ parametrizes Kepler ellipses and its exterior  $\|(a,b,c)\|^2>0$ parametrizes Kepler hyperbolas. See Figure \ref{fig:conics}. 

\begin{figure}[!htb]
        \centering
        \includegraphics[width=\textwidth]{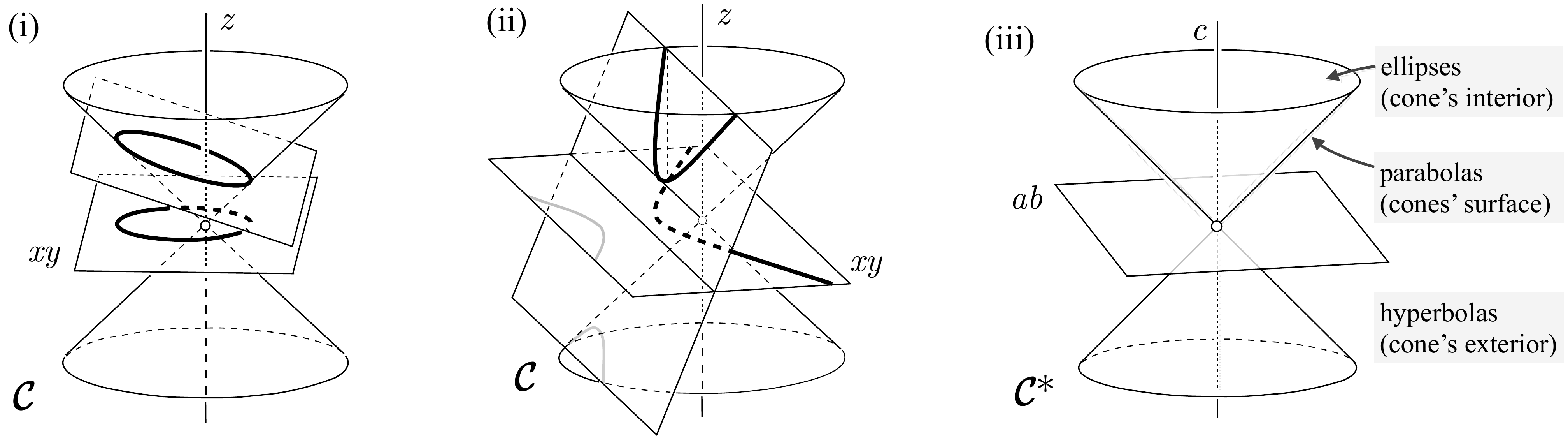} \qquad
                \caption{Kepler orbits  are orthogonal projections of conic  sections. (i) Ellipses. (ii) Hyperbolas. (iii) The space of Kepler orbits.}
        \label{fig:conics}
  \end{figure}

The orbital symmetries of Theorem \ref{thm:main1} clearly act on the space of Kepler orbits and thus on $\Rot_+$. Again, this is only a local action (a 7-dimensional Lie algebra of vector fields), but it extends to a global action on all of $\Rto$.

%\note{GB: edited the statement of Thm 2 and some of the preceding text.}
\begin{theorem}\label{thm:main2}The local group action of the orbital symmetries of the Kepler problem on $\Rto_+$ extends to $\Rto$, generating the identity component of the  group   $\CO_{2,1}\ltimes\Rto$ of Minkowski similarities (compositions of Minkowski  rotations, dilations and translations).  The infinitesimal generators of this action, corresponding to those of Equation \eqref{eq:vf}, 
are 
\be\label{eq:dualvf} 
-a \partial_a-b \partial_b-c \partial_c,\
-b \partial_a+a \partial_b,\
-a \partial_c-c \partial_a,\
-b \partial_c-c \partial_b,\
\partial_a,\ \partial_b\ ,\partial_c.
\ee
\end{theorem}

The first vector field generates dilations in $\Rto$, the next 3 generate Minkowski rotations
 about the origin  and the last 3 generate translations. It follows that orbital symmetries actually `mix' the orbit types (ellipses, parabolas, hyperbolas).
 
%\note{GB: added this paragraph}  
The horizontal plane $\{c=0\}\subset\Rto$ corresponds to `ideal' Kepler orbits which are inevitably added upon completing the orbital symmetry action. For $(a,b,0)\neq (0,0,0)$ they are (affine) lines in $\R^2\setminus 0$, obtained by projecting to the $xy$ plane sections of $\cC$ by vertical affine 2-planes in $\R^3$. The point $(0,0,0)\in\Rto$ corresponds to the `line at infinity' in the Kepler plane.  

\ssec{Sketch of proof of Theorems \ref{thm:main1} and \ref{thm:main2}}\label{sec:sketch}
With Figure \ref{fig:conics} in mind, consider the group $\COto\subset\GL_3(\R)$, preserving the quadratic form $x^2+y^2-z^2$ up to scale. Its identity component   acts on  $\cC_+:=\{x^2+y^2=z^2, z>0\}$, preserving its set of plane sections, thus projects  to  an action on $\R^2\setminus 0$ by orbital symmetries.  This accounts for the first 4 vector fields of Equation \eqref{eq:vf}.

Next, consider the 3-dimensional projective space $\RP^3$ with homogeneous coordinates $(X:Y:Z:W)$ and embed $\R^3\hookrightarrow\RPt$ as the affine chart $W\neq 0$, $(x,y,z)\mapsto (x:y:z:1).$ The closure of $\cC$ in $\RPt$,  $\ccC=\{(X:Y:Z:W)\st X^2+Y^2=Z^2\}$, is obtained by adding to  $\cC$ the `circle at infinity' $S^1_\infty=\{X^2+Y^2=Z^2, W=0\}$. See Figure \ref{fig:cylinders}. Now consider the group  $\tOS\subset \GL_4(\R)$,  preserving the (degenerate) quadratic  form $X^2+Y^2-Z^2$, up to scale. A simple calculation (see Section \ref{sec:proofs} below) shows
that  $\tOS$ is an 8-dimensional group, thus its image $\OS=\tOS/\R^*\subset \PGL_4(\R)$ is $7$-dimensional, acting effectively on $\ccC$, preserving its set of (projective) plane sections. It leaves invariant the set of sections by planes {\em not} passing through the vertex of $\ccC$, parametrized by $\Rto$. 
The action restricts to a local action on $\cC_+\subset \ccC$, then projects to a local  action on $\R^2\setminus 0$ by orbital symmetries. Equations \eqref{eq:vf} and \eqref{eq:dualvf}  
follow easily from this  description.   

%\note{GB: added `reviewed in the Appendix''}
Finally, we use a basic result of Lie's theory of symmetries  of ODEs (reviewed in the Appendix), according to which the maximum dimension of the group of  point symmetries of a 3rd order ODE is 7, thus the above construction provides the full group of orbital symmetries of the Kepler problem.  See Section \ref{sec:proofs} below for the full details.

\begin{figure}[!htb]
        \centering
    \includegraphics[width=\textwidth]{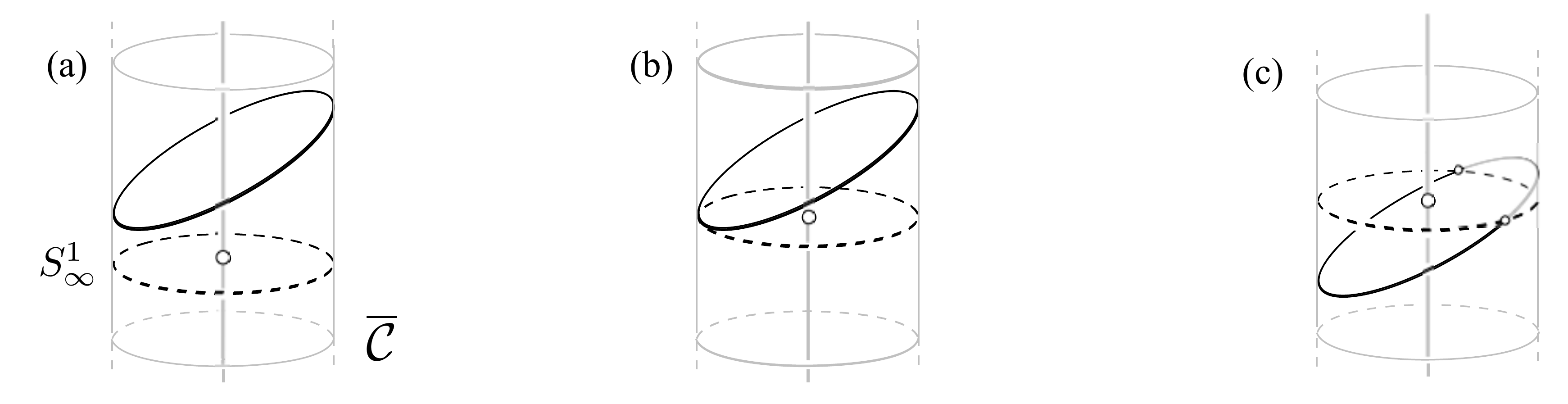}
        \caption{ In the  affine chart $\{Z\neq 0\}\subset \RP^3$, the  cone  $\ccC=\{(X:Y:Z:W)|    X^2+Y^2=Z^2\}\subset\RP^3$ appears as an infinite vertical cylinder, its   vertex $(0:0:0:1)$ is `at infinity' and  the `circle at infinity' $S^1_\infty=\ccC\cap\{  W=0\}$ is visible (the dotted horizontal circle). A   conic section may intersect $S^1_\infty$ in 0,  1 or  2 points; the  corresponding Kepler orbit is an ellipse, parabola or hyperbola, as in  figures (a), (b) or (c), respectively.  In the hyperbolic case (c), $S^1_\infty$ divides the section into two `branches'; the Kepler orbit corresponds to the branch whose convex hull intersects the vertical axis (the dark arc of the solid ellipse). }
        \label{fig:cylinders}
  \end{figure}

\ssec{2-parameter subfamilies} 
The simplest example of a 2-parameter family of plane curves  (also called a `path geometry')  is the family of straight lines. It admits an 8-dimensional local group of symmetries (the projective group), the maximum dimension possible for a 2-parameter family of plane curves. A 2-parameter  family of plane curves locally diffeomorphic to this family is called {\em flat}. There are no straight lines among Kepler orbits, but there are flat 2-parameter subfamilies.  

\begin{theorem}\label{thm:parab}Kepler's parabolas form a  flat 2-parameter family of curves. The map $\z\mapsto \z^2$ (in complex notation) is a local diffeomorphism  taking straight affine lines to Kepler parabolas. 
\end{theorem}

This theorem is essentially known. The  squaring map $\z\mapsto \z^2$, in the context of the Kepler problem, is  known sometimes as the  {\em Levi-Civita} or  {\em Bohlin} map. It can  be  used to define a local orbital equivalence between  Hooke and Kepler orbits (see e.g. Appendix 1 of  \cite{Ar}).

%\note{GB: added $\pm M$}
\begin{theorem}\label{thm:angmom}Kepler's orbits with fixed angular momentum  $\pm M\neq 0$ form a flat 2-parameter family of curves. The map $\r\mapsto \r/(1-r/M^2)$ takes  Kepler orbits with angular momentum $M$ to straight  lines. 
\end{theorem}

See Section \ref{sec:reminder} for a reminder about the angular momentum   (also Figure \ref{fig:kepler_orbits}). The proof of this theorem is particularly simple using the geometry of the space $\Rto$ of Kepler orbits: the family of Kepler orbits %\note{GB: replaced $M$ by $|M|$}
with fixed $|M|$ is represented  in $\Rto$ by a horizontal plane; a vertical translation in this space, which according to Theorem \ref{thm:main2} is available as an orbital symmetry, maps this plane to the plane $c=0$, parametrizing lines in the $xy$-plane. 

Next we consider Kepler orbits with fixed energy $E\neq 0.$  These fill up a plane region $\cH_E$, the {\em Hill region}. For $E\ge 0$ (Kepler hyperbolas with  major axis  $1/E$ or Kepler parabolas) the Hill region  is  the whole punctured plane,  for $E<0$ (Kepler ellipses with  major axis $1/|E|$) it is a punctured disk of radius $1/|E|$. See Figure \ref{fig:hill}.

\begin{figure}[!htb]
        \centering
    \includegraphics[width=\textwidth]{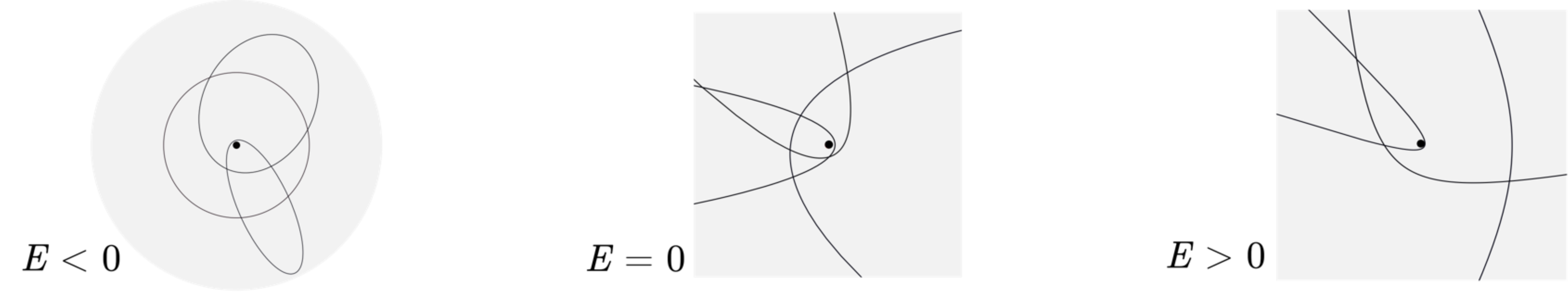}
        \caption{Kepler orbits of fixed energy $E$ fill up the Hill region $\cH_E$. 
        }
        \label{fig:hill}
  \end{figure}

\begin{theorem}\label{thm:hill1} \benum
\item For each  fixed energy $E\neq  0$,  the 2-parameter family of Kepler orbits with energy $E$ is non flat but is locally homogeneous: its orbital symmetry group is a 3-dimensional subgroup of the 7-dimensional group of Kepler's orbital symmetries, isomorphic to $\PSLt$ and generated by the infinitesimal symmetries  
\be\label{eq:fe}\partial_\theta, r(\partial_x + Ex\partial_r), r(\partial_y +E y\partial_r).
\ee

\item  For $E<0$ the action of $\PSLt$ on the Hill region $\cH_E$ is global; for $E>0$  it is only local. 
\end{enumerate}
\end{theorem}
%\note{added this paragraph}
This theorem is also essentially known, or at least can be deduced  easily by experts on `superintegrable metrics' from known  results from  the 19th centrury by S. Lie and G. Koenigs  (see for example \cite{BMM} and references within; we thank V. Matveev for pointing out to us this relation).

 Our proof  of this theorem is  quite simple   using the geometry of the space of orbits $\Rto$: as we explain in Section \ref{sec:reminder}, orbits of fixed energy  $E$ correspond to one of the sheets of the hyperboloid of two sheets $a^2+b^2-(c-|E|)^2=-E^2$ (the upper sheet for $E<0$, the lower one for $E>0$). See Figure \ref{fig:parabo}(iii). The Minkowski metric in $\Rto$ restricts to a hyperbolic metric in each of these sheets,  the subgroup of 
   $\OS\simeq \CO_{2,1}\ltimes\Rto$ preserving the hyperboloid acts as the full group of isometries of this metric, with generators given by Equation \eqref{eq:fe}. 
   
\begin{figure}[!htb]
        \centering
        \includegraphics[width=\textwidth]{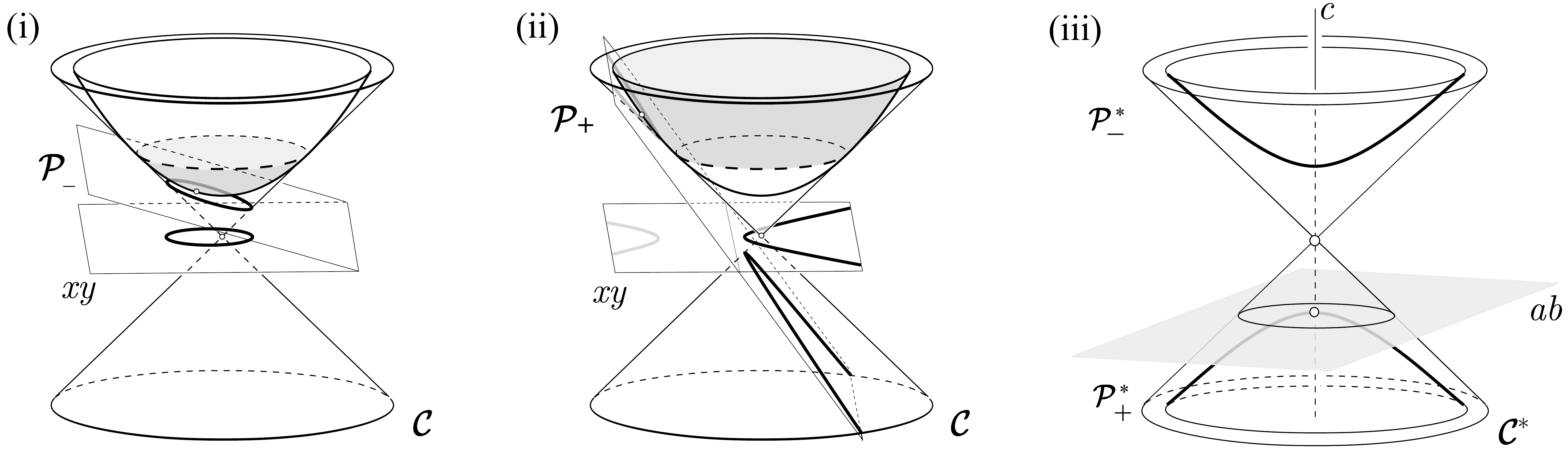} \qquad
                \caption{ Theorem  \ref{thm:hill1}. 
                (i) Kepler ellipses of fixed (negative) energy  correspond to sections of  $\cC$ by planes tangent to the lower part  of a fixed paraboloid of revolution $\cP$ inscribed in $\cC$. (ii) Kepler hyperbolas  of the opposite (positive) energy correspond to sections  of $\cC$ by planes tangent to the upper  part of $\cP$. (iii) The surface dual to $\cP$  is  a 2-sheeted hyperboloid of revolution  $\cP^*\subset \Rto$ tangent to the $c=0$ plane. Its upper and lower sheets correspond  to $\cP_-$ and  $\cP_+$, respectively. 
                }
        \label{fig:parabo}
  \end{figure}

Any two Hill regions  with the same sign of energy are obviously orbitally equivalent by dilation. For opposite signs of energies this is still true but less obvious. 
\begin{theorem}\label{thm:hill2} 
 $\cH_1$ is orbitally embedded in $\cH_{-1}$ by the map  $\r\mapsto \r/(1+2r).$ See Figure \ref{fig:emb}. 
\end{theorem}

Viewed in $\Rto$, where the two Hill regions correspond to the two sheets of a hyperboloid, the map is simply the reflection about a  horizontal plane $c=1$, interchanging the two sheets. See Figure \ref{fig:parabo}(iii).

\begin{figure}[!htb]
        \centering
        \includegraphics[width=\textwidth]{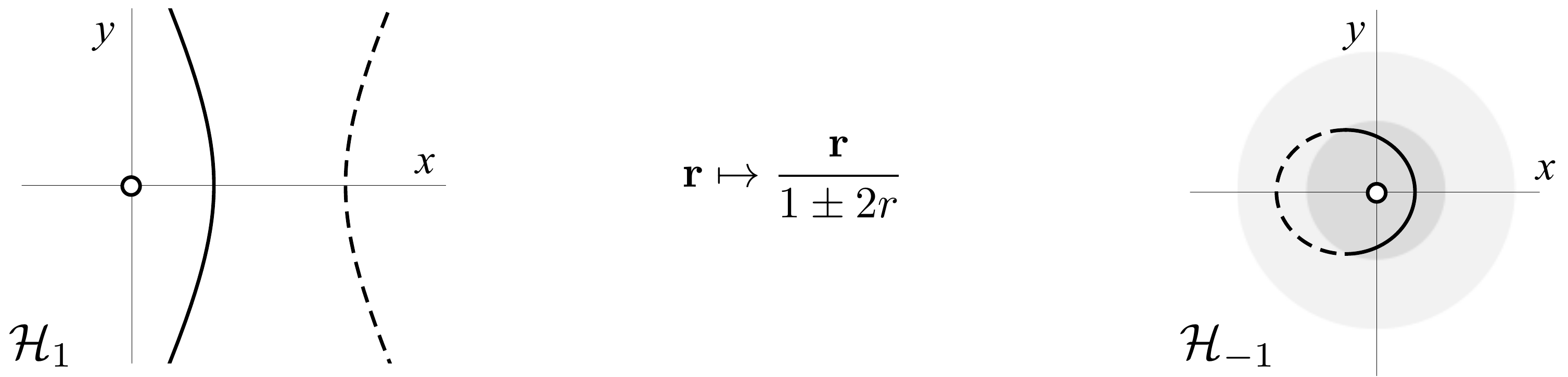} 

        \caption{
        Theorem \ref{thm:hill2}. The image of $\cH_1=\R^2\setminus 0$ (left) under $\r\mapsto\r/(1+2r)$ is the darker punctured disk of radius 1/2 in $\cH_{-1}=\{0<x^2+y^2<1\}$ (right). Each hyperbolic  orbit in $\cH_{1}$ (the solid curve) is mapped onto `one half' of an elliptic orbit in $\cH_{-1}$. The map, $\r\mapsto\r/(1-2r)$ maps $\cH_1$ onto the annulus $1/2<x^2+y^2<1$ in $\cH_{-1}$, taking   the `repulsive' branch  (dotted) onto the other half of the ellipse.       
        }
        \label{fig:emb}
  \end{figure}

\ssec{Further results}\label{sec:apps}
\begin{enumerate}[left=0px, itemsep=-3px]
\item We establish a dictionary between  the Minkowski geometry of  the Kepler orbit space $\Rto$ and properties  of  Kepler orbits. For example: a parabolic (or isotropic) plane in $\Rto$ corresponds to the family of Kepler orbits passing through a fixed point. See Table \ref{tab:pd}  of Section \ref{sec:dual}.

 \item We give three illustrations of the usage of this  dictionary: a new proof  of `Kepler's fireworks' (Proposition \ref{prop:fireworks}), a Keplerian analogue of the   4 vertex and  Tait-Kneser theorems (Theorem \ref{thm:TK}) and  a `minor axis version' of Lambert's Theorem  (Theorem \ref{thm:Lambert}).

\item Similar results  to Theorems \ref{thm:main1}-\ref{thm:hill2} hold for orbital symmetries  of  the Hooke problem  -- the set of conics sharing a center (trajectories of mass points under central  force proportional to the distance to the origin), and the orbits of the corresponding `Coulomb' problems, where the sign of the force is reversed, becoming a repelling force. By central projection, our results extend to Hooke and Kepler orbits on surfaces of constant curvature (sphere and  hyperbolic plane). See Table \ref{tab:proj}.

\item We establish a converse to Theorem \ref{thm:main1}:  among all  central forces, Hooke and Kepler force laws are the only ones producing `flat' families of orbits (3 parameter families with a 7-dimensional group of symmetries). See Theorem \ref{thm:flat}. This is reminiscent of  Bertrand's Theorem (1873), characterizing these two  force laws  as the only central force laws  with bound orbits  all of whose bound orbits are  closed \cite{Be}, \cite[page 37]{Ar1}.

\end{enumerate}

\mn \centerline{*\qquad *\qquad *}

 \sn{\bf Techniques.} Other than standard projective and differential geometric constructions, we  use some of the  work of  S.~Lie (1874), A.~Tresse (1896)  and K. W\"unschmann (1905) on point symmetries of  2nd and 3rd order ODEs.  We do not assume the reader's familiarity with  their work. We  summarize in the Appendix  the needed tools of this theory. 

 \sn{\bf Figures.} The figures here were computer generated using Wolfram's Mathematica and Apple's Keynote.

 \mn{\bf Acknowledgment.} 
 %\note{added Matveev}
 We thank Richard Montgomery, Sergei Tabachnikov, Alain Albouy and Vladimir Matveev for fruitful correspondence and discussions. GB was supported by CONACYT Grant A1-S-4588. 

\section{Wider context: `orbital' vs `dynamical' symmetries}
The  Kepler problem is centuries old with an enormous literature.  It is hard to imagine one can add anything new to this problem  in the 21st century.  Yet, new and interesting works continue to appear. See, for example, \cite{AV, Bl, Mo, So, Giv, GS}. 
Some  facts have been  rediscovered several times, centuries apart, especially before the existence of internet search engines. For example,  V.I. Arnol'd  attributes in his  1990 book \cite[Appendix 1]{Ar} the fact that $\z\mapsto \z^2$ maps Hooke orbits to Kepler orbits to  K. Bohlin's 1911 article \cite{B},  then goes on to generalize it to a `duality' between central force power  laws. In fact, all this  appeared in C. Maclaurin's 1742    `Treatise of fluxions'   \cite[Book II, Chap. V,  \S 875]{M} (we thank S. Tabachnikov for pointing out this reference to us). 

One of the most studied  aspects of the Kepler problem are its symmetry properties.  The most obvious symmetries are diffeomorphisms of the plane,  mapping solutions $\r(t)$ of the underlying ODE,  $\ddot\r=-\r/r^3$, to solutions. One can show that these consist only of the rotations about the origin and reflections about lines through the origin, valid for any central force motion. 

More interesting symmetries arise 
when the Kepler problem is considered as a Hamiltonian system, ie a flow defined on its phase space $T^*(\R^2\setminus 0)$. The symplectomorphisms of phase space preserving {\em parametrized} trajectories of this flow form a larger group of symmetries, associated to the Hamiltonian flows of additional conserved quantities such as components of the Laplace-Runge-Lenz vector.  These symmetries generate a (local)  $\SO_3$-action on the open subset of phase space with negative energy. Apart from the lift of the rotation symmetries above, these oft-called `hidden' symmetries do not descend to  an action on the Kepler plane, even locally. The action is rather on phase space, mixing position and momentum variables.
 A good reference for this type of `dynamic' or `phase space' symmetries of the Kepler problem is the book  \cite{GS} or Chapters 3 and 4 of \cite{FvK}.

 In contrast, the symmetries  in this article are  `orbital' symmetries, acting on the  {\em configuration space} of the Kepler problem, $\R^2\setminus 0$, {\em not its phase space}. They are closer to the symmetries one can extract from Albouy's `projective dynamics' papers \cite{AlbPr0, AlbPr}. 
 
So how  original are our results? As far as we can tell, after consulting with experts and searching the literature, our results  are  new.   The articles \cite{AlbPr0, AlbPr, Car} are the nearest in spirit that we found. `Hidden symmetries'  of the Kepler problem, ie of its phase space, have been studied extensively, and  symmetries of 2nd and 3rd order ODEs have been studied extensively as well since the mid 19th century, but it seems that the symmetries of the 2nd and 3rd order ODEs that arise in the Kepler problem have not been studied systematically before, which is the present article's contribution. 

But of course, given the  subject's long and rich history,  it is is still quite possible that at least some of the theorems announced here have been noted before, in some form or another. If some of the readers of this article are aware of such work we will be grateful if they  contact us.

\section{A reminder on the Kepler problem}\label{sec:reminder}
Here we review briefly some  well known facts  about the Kepler problem that will be used in the sequel. See also  \cite{AlbLect,Ar,Ar1,Giv}.

Kepler  orbits are the   {\em unparametrized} plane curves traced by  the solutions of
the ODE
\be \label{eq:kepler}
\ddot \r=-{\r\over r^3},
\ee
where  $ \r=\r(t)=(x(t), y(t))\in\R^2\setminus 0$ and $ r:=\|\r\|=\sqrt{x^2+y^2}.$

The {\em energy}  and {\em angular momentum} of a solution are  
\be\label{eq:conserve}
E:={1\over 2}\|\dot \r\|^2-{1\over r}, \quad M:=x \dot y-  y \dot x,
\ee
 respectively, and can be easily shown to remain constant during the motion. 
 
 Note that 
$M$ is twice the {\em sectorial velocity}, the rate at which  area is swept by the line segment connecting the origin to $\r(t)$. It follows that  $M=0$ if and only if the motion is along a line passing through the origin. Our exclusion of `collision' orbits thus amounts to assuming $M\neq 0.$ 
%\note{GB: added a note about $E$, $M$ being parametrization independent}
Note also that although $E$ and $M$ are defined in equation \eqref{eq:conserve} via the time parametrization  of the Kepler orbit, they are in fact determined by the shape of the underlying unparametrized curve (except for the sign of $M$). 
%
%\note{GB: added reference to       Fig  \ref{fig:kepler_orbits}}
See Figure \ref{fig:kepler_orbits}.

A {\em conic} in a Euclidean plane  is the locus of points with  constant  ratio of distances to a fixed  point  and a fixed line. The fixed point,  line and ratio are  called a  {\em focus, directrix} and {\em eccentricity} $e$   (respectively). 
Conics with $e>1$, $e=1$, $0< e <1$ and $e=0$ are hyperbolas, parabolas,  non-circular ellipses and circles (respectively).

Identify   the $xy$ plane  with  the 
plane $z=0$ in   $\R^3$, $(x,y)\mapsto (x,y,0)$. We use the term  `projection' to mean the orthogonal projection $\R^3\to\R^2,$ $(x,y,z)\mapsto (x,y)$.

 \begin{theorem}   \label{thm:EM}
 \benum
 \item Every  Kepler orbit is  the projection   of   a section of  the cone $\cC=\{x^2+y^2=z^2\}\subset \R^3$ by a plane $ax+by+cz=1$, $c\neq 0.$ 
 
 More precisely: if $c>0$ then  the orbit is the projection of the intersection of the plane with the  upper cone $\cC_+:=\cC\cap \{z>0\}$;  if $c<0$ then it is the projection of the intersection of the plane with the  lower cone $\cC_-:=\cC\cap \{z<0\}$. 
 
  \item The projected section is a conic with  a focus at the origin and eccentricity 
\be\label{eq:eccen}
e={\sqrt{a^2+b^2}\over| c|}
\ee

\item  The angular momentum and energy of the projected Kepler orbit are
\be \label{eq:EM}%e={\sqrt{a^2+b^2}\over c}, \quad  
M= \pm {1\over \sqrt{|c|}}, \quad 
E={a^2+b^2-c^2\over 2|c|}.
\ee

 \end{enumerate}

  \end{theorem}
 
 \begin{remark}\label{rmrk:ojo} For positive energy orbits (hyperbolas), the plane section has two components (branches), one in each of $\cC_\pm$, and one needs to pick carefully the correct branch, as stated in item (a). 
 \end{remark}

\begin{proof} 
(a)  Let $\r(t)=(x(t), y(t))$ be a solution of Equation \eqref{eq:kepler} with $M\neq 0$. Rewriting Equations \eqref{eq:kepler} and \eqref{eq:conserve} in polar coordinates, we have 
\be\label{eq:polar} \ddot r=-{1\over r^2}+{M^2\over r^3}, \quad E={\dot r^2\over  2}-{1\over r}+{M^2\over 2r^2 }.%,    \quad M=r^2\dot\theta.
\ee
 From the first equation  follows that the inhomogeneous linear ODE
 $$\ddot f+{f\over r(t)^3}={M^2\over r(t)^3}$$ 
 has two particular solutions: $r(t)$ and the constant solution $M^2$. Their  difference is  thus a solution of the homogeneous equation $\ddot f+f/ r(t)^3=0$. But $x(t), y(t)$ are two solutions of this equation, linearly independent for $M\neq 0$, hence  there are constants $A,B$ such that  $r(t)-M^2=Ax(t)+By(t).$ Rearranging and renaming the constants we obtain $ax+by+cr=1$, $r^2=x^2+y^2$, %with $z=r>0,$ $z^2=x^2+y^2,$ $c=1/M^2>0,$ 
as claimed. % Clearly, projecting the section of $\cC_-$ by  $ax+by-cz=1$ gives the same orbit. 

The statement about the precise right half cone to  pick  is best  seen by examining   Figure \ref{fig:conics}, (i) and (ii). 
 
\mn (b)   By rotating the secting plane about the $z$ axis and possibly reflecting it about the $xy$ plane, we can assume $a\geq 0, b=0, c>0$. If $a=0$ then the secting plane is parallel to the $xy$ plane and the projected conic is a circle ($e=0$). Otherwise,  $a>0$, the  secting plane    is $ax+cz=1$,  its  intersection with the $xy$ plane is  the line $ax=1$ and the  projected conic  is $ax+cr=1$.  The   ratio of distances  of  a point $(x,y)$ on the projected section to  the origin and the intersection line is thus  $e=r/|x-1/a|=ar/|cr|=a/c$, a constant, hence the projected section is a conic, the origin is a focus and the intersection line is the corresponding  directrix.   The formula for $e$ follows from this calculation, since rotation of the plane $ax+by+cz=1$ about the $z$ axis and reflecting it about the $xy$ plane does not affect  the values of $e,|c|$ and  $a^2+b^2$.  
 
\mn (c)  The formula for $M$ follows from the  proof of  item (a). For $E$, we  again assume   $a\geq 0, b=0, c>0.$ The  orbit is then $ax+cr=1$ and at the pericenter  (the point nearest the origin)  we have $x=r=1/(a+c)$. Using this in the formula for $E$ in Equation   \eqref{eq:polar}, with $\dot r=0$, $M^2=1/c$,  we get $E=(a^2-c^2)/(2c).$ For a  general secting plane  $a^2$ is  replaced  with $a^2+b^2$ and $c$ with $|c|$. 
\end{proof}

\begin{remark}The  clever argument  in the proof of item (a) above is due to Lagrange \cite{L}. For a more geometric proof of item (c) see   \cite[\S4]{Giv}.
 \end{remark}

\begin{corollary}\label{cor:dc} The cone $\{a^2+b^2=c^2\}\subset\Rto$ parametrizes Kepler parabolas, its interior $a^2+b^2<c^2$ Kepler ellipses and exterior $a^2+b^2>c^2$ Kepler hyperbolas. See Figure \ref{fig:conics}(iii). 
\end{corollary}

\begin{corollary}\label{cor:M}
 Kepler orbits with  angular momentum $M\neq 0$  have fixed {\em latus rectum} $2M^2$ and are the projections of  sections of $\cC$ by non-vertical planes passing  through $(0,0,  M^2)$ or  $(0,0,  -M^2)$. See Figure \ref{fig:EM_pencils}(a). 
\end{corollary}

This is immediate from Equations  \eqref{eq:eccen} and \eqref{eq:EM}.

\begin{figure}[!htb]
        \centering
        \includegraphics[width=.8\textwidth]{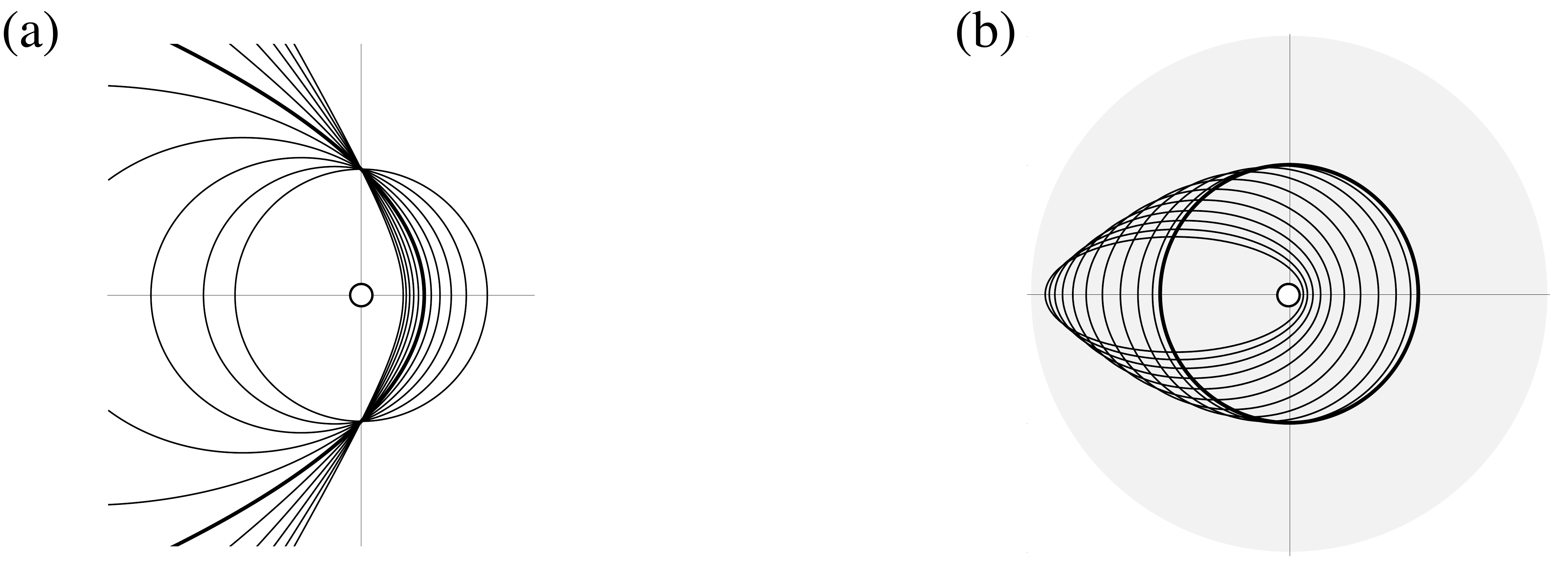} \qquad
                \caption{(a) Kepler orbits with  fixed angular momentum  (same as fixed {\em latum rectum}); the  heavy curve is a  parabola.  (b)  Kepler ellipses with fixed (negative) energy $E$ (same as fixed major axis). 
                }
        \label{fig:EM_pencils}
  \end{figure}

\begin{corollary}\label{cor:E}
Kepler orbits with   energy $E\neq 0$  are the projections  of  sections of $\cC$ by  planes tangent to the   paraboloid of revolution
$$\cP:=\left\{(x,y,z)\in \R^3\st z={|E|\over 2}\left(x^2+y^2\right)+{1\over 2|E|}\right\},$$
inscribed in  $\cC$ and tangent to it  along a horizontal circle, dividing $\cP$   into two components:  
Kepler ellipses with energy $-|E|$ are the projections of sections of $\cC_+$ by planes tangent to  the lower component $\cP_-=\cP\cap\{z<1/|E|\}$; Kepler hyperbolas with energy $|E|$ are the projections of sections of  $\cC_-$ by planes tangent to the upper component  $\cP_+=\cP\cap\{z>1/|E|\}$.  See Figure \ref{fig:parabo}. 
\end{corollary}

\begin{proof}$\cP$ 
is given in homogeneous coordinates $(X:Y:Z:W)$ on $\RP^3$  by $E^2(X^2+Y^2)-2|E|ZW+W^2=0.$ The dual equation, parametrizing the  planes $AX+BY+CZ+DW=0$ tangent to $\cP$,   is given by inverting the coefficient matrix of the quadratic equation defining $\cP$, and is $A^2+B^2-C^2-2|E|CD=0,$ or  in affine coordinates, $a^2+b^2-c^2+2|E|c=0$. At a point $\p_0=(x_0,y_0,z_0)\in\cP$ the tangent plane is $ax+by+cz=1$ where $(a,b,c)=(|E|x_0, |E|y_0, -1)/(|E|z_0-1)$. If $\p_0\in \cP_-$ then  $z_0<1/|E|$ hence  $c>0$, so  by Equation \eqref{eq:EM} the energy of the corresponding orbit is $(a^2+b^2-c^2)/(2c)=-|E|,$ as needed. A similar calculation  for 
the case $\p_0\in \cP_+$ completes the proof. \end{proof}

\begin{remark}The last corollary we learned  from \cite[page 145]{Giv}, although our proof is quite different. 
\end{remark}

\section{The geometry of the space of Kepler orbits}
\label{sec:dual}

Recall that   $\Rto$ is the 3-dimensional space with coordinates $a,b,c$, equipped with the indefinite quadratic form $\|(a,b,c)\|^2:=a^2+b^2-c^2$ and associated  flat Lorentzian metric $%\langle \cdot, \cdot \rangle:=
ds^2=da^2+db^2-dc^2$.  A line in $\Rto$ is {\em spacelike}, {\em null} or {\em timelike} if $ds^2$ restricts on it to a positive, null or negative metric, respectively. A plane in $\Rto$ is {\em elliptic}, {\em parabolic} \footnote{Some authors use the term `isotropic' instead of `parabolic'. For example,  \'E. Cartan \cite{Cart}. Elliptic planes are called also `spacelike.'} or  {\em hyperbolic} if $ds^2$ restricted to it is of  signature $(2,0),(1,0),$ or $(1,1)$, respectively. The {\em null cone} with vertex $\v\in\Rto$ is the set of points $\v'\in\Rto$ such that $\|\v-\v'\|^2=0$; equivalently, the  union of null lines through $\v$.

\subsection{Duality}

The equations $ax+by+cz=1, x^2+y^2=z^2$ define a duality between Kepler's $xy$ plane and  Minkowski's space $\Rto$: to each point $(a,b,c)\in\Rto\setminus 0 $ corresponds a curve in the $xy$ plane, a Kepler orbit if $c\neq 0$ or a straight line if $c=0$, the projection of the intersection of the  plane  $ax+by+cz=1$ with one of the components of  $\cC=\{x^2+y^2=z^2\}$  (see Theorem \ref{thm:EM}(a)): if $c>0$ then one projects the intersection with $\cC_+=\cC\cap \{z>0\}$,  if $c< 0$ the intersection with $\cC_-=\cC\cap \{z<0\}$ and  if $c=0$ the intersection with either component. Conversely, to each point $(x,y)\in \R^2\setminus 0$ corresponds the  plane  $ax+by+ cr=1$  in $\Rto$,  where $r=\sqrt{x^2+y^2}. $  Table \ref{tab:pd} summarizes some instances of this duality. 
%\note{added $\pm$ in item \ref{it:fam} of the table}

%\begin{remark}The annoying sign ambiguity in the above description can be eliminated by viewing the $xy$ plane as a local chart in the projective cone $\overline\cC\subset\RP^3$. This is explained in the proof of Theorem \ref{thm:main1} below. \end{remark}

%%%%%%%%%%%%%%%%%%

%%%%%%%%%%%%%%%%%

\begin{table}[!htbp]%[ht]
%\vspace{-2cm}
\renewcommand{\arraystretch}{2.2}

\captionsetup{ font=normalsize}
\caption{Kepler-Minkowski   duality}
{
%\footnotesize
%\small
\centering
\hspace{-.5cm}
 \begin{tabular}[t]{
% p{.01\textwidth}
c
 p{.48\textwidth}|p{.48\textwidth} } 
& Kepler $xy$-plane& Minkowski space   $\Rto$\\ 
\cmidrule(lr){2-3}
\ti&A Kepler orbit (or a line)   & A point  \\
\ti&A Kepler ellipse/parabola/hyperbola & A point inside/on/outside  $a^2+b^2=c^2$\\

\ti&A line  & A point in the $ab$ plane \\
\ti{\label{it:para}}& A point & A parabolic  plane\\
\ti{\label{it:nullcone}}&Kepler orbits tangent to a given  Kepler orbit & The  null cone with a given vertex \\

\ti{\label{it:nullline}}&Kepler orbits tangent at a point& A null line \\

\ti{\label{it:spaceline}}&Kepler orbits passing through 2 given points& A spacelike line \\
%\ti&Nested Kepler orbits & A timelike curve \\

\ti{\label{it:timeline}}&Nested Kepler orbits with concurrent directrices& A timelike line \\

\ti{\label{it:fam}}&Kepler orbits of fixed angular momentum $\pm M\neq 0$ & A horizontal plane $c\neq 0$
\\

\ti{\label{ti:ee}}&Kepler ellipses with energy $E<0$ (projected  sections of $\cC$ by planes tangent to $2|E|z=E^2\left(x^2+y^2\right)+1$, $|E|z<1)$ 
& The upper  sheet of  the  hyperboloid of 2 sheets $a^2+b^2-(c-|E|)^2=-E^2$\\
\ti{\label{ti:eh}}&Kepler hyperbolas of  energy $E> 0$ (projected of sections of $\cC$ by planes tangent to $2|E|z=E^2\left(x^2+y^2\right)+1$, $|E|z>1)$  
& The lower  sheet of  the  hyperboloid of 2 sheets $a^2+b^2-(c-|E|)^2=-E^2$\\

\ti{\label{ti:me}}&Kepler ellipses with minor axis $B$ (projected   sections of $\cC$ by planes tangent to 
$x^2 + y^2  - z^2= -B^2/4$) & The hyperboloid of 2 sheets $a^2+b^2-c^2=-4/B^2$\\

\ti{\label{ti:mh}}&Kepler hyperbolas with  minor axis $B$: projected  sections of $\cC$ by planes tangent to $x^2 + y^2   -z^2=B^2/4$ & The hyperboloid of 1 sheet $a^2+b^2-c^2=4/B^2$\\

\ti{\label{ti:cu}}&Central projections of  Kepler orbits with energy $\pm E_k$ on a surface of constant curvature $k$ &  
The hyperboloid (of 1 or 2 sheets, depending on $k$) $a^2 + b^2 - (c - |E_k|)^2=- E_k^2 - k $  \\

 \end{tabular}
}
\label{tab:pd}
\end{table}

We shall not dwell on all items of this table, as most  reflect  statements proven elsewhere in this article or are  simple to verify. We  sketch here proofs of a few  items and leave the rest  for the reader to explore.

\begin{proposition}[Item \ref{it:para} of Table \ref{tab:pd}]
The set of Kepler orbits sharing a point corresponds to a  parabolic plane in $\Rto$. Every parabolic  plane in $\Rto$ arises in this way. 
\end{proposition}

\begin{proof} A plane $ax+by+cz=1$ in  $\Rto$ is parabolic  if and only if  it forms an angle    of 45 degrees with a horizontal  plane.  This angle satisfies  $\tan\alpha=\sqrt{x^2+y^2}/|z|$ and the result follows. 
\end{proof}

 \begin{remark} This last proposition is equivalent to Corollary \ref{cor:dc} above by projective duality.
 \end{remark}
 
\begin{proposition}[Item \ref{it:nullline} of Table \ref{tab:pd}] \label{prop:null}
The set of Kepler orbits tangent to a  given Kepler orbit  at one of its points  corresponds to a null line in $\Rto$. Every null line is obtained in this way. See Figure \ref{fig:null_pencil}. 
\end{proposition}
\begin{proof} Let $C$ be the given Kepler orbit and $P\in C$. Using Kepler's orbital symmetries (Theorems \ref{thm:main1} and \ref{thm:main2}) we can assume, without loss of generality, that $C$  is the unit circle and $P=(0,1)$ (see  Remark \ref{rmrk:alert} below, though). A Kepler orbit  $ax+by+cr=1$ is tangent to $C$ at $P$   if and only if $a=1, b+c=0$, which is a null line in $\Rto$. Every null line is congruent to this line by an orbital symmetry.
\end{proof}

\begin{figure}[!htb]    
\centering\includegraphics[width=\textwidth]{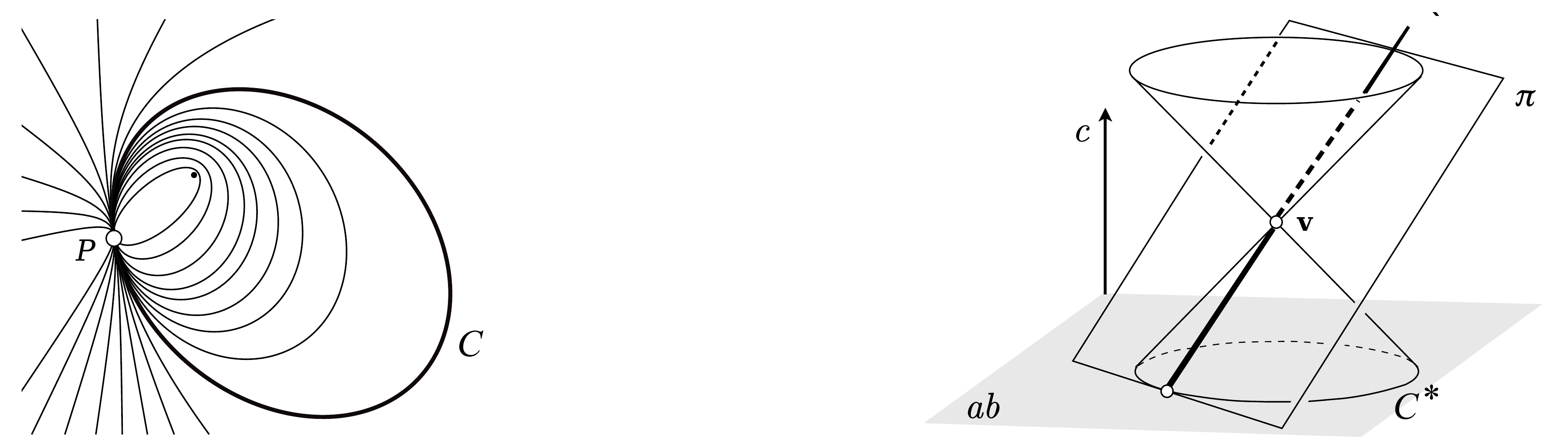}
 \caption{\small Proposition \ref{prop:null}. 
 Left: the set of Kepler orbits tangent  to a fixed Kepler orbit $C$ (the dark curve) at a fixed point $P\in C$. 
Right: the point $\v \in \Rot_+$ corresponds to $C$, the null cone with vertex $\v$ corresponds to  all Kepler orbits tangent $C$, the parabolic plane $\pi$ corresponds to all Kepler orbits passing through the  point  $P\in C$. The intersection of $\pi$ with the cone is one of its generators, a null line, corresponding to the Kepler orbits tangent to $C$ at $P$.   The intersection  of the cone with the $c=0$ plane is a circle $C^*$, corresponding to all lines tangent to $C$ (see Proposition \ref{prop:dualC}).  }
 \label{fig:null_pencil}
\end{figure}

\begin{proposition}[Item \ref{it:timeline} of Table \ref{tab:pd}] \label{prop:time}
The Kepler orbits corresponding  to a line in $\Rto$ (a `pencil' of Kepler orbits) have concurrent directrices (they all pass through a single point). The orbits of a timelike pencil are nested (same as disjoint). \end{proposition}
\begin{proof} The orbits of a Kepler pencil corresponding to a line $\ell^*\subset \Rto$ are obtained by projecting sections of $\cC$ by planes passing through a fixed line  $\ell\subset \R^3$ (the line dual to $\ell^*$). The directrix of a Kepler orbit is the intersection of the secting plane with the $xy$ plane. Thus all directrices of Kepler orbits in a pencil pass through a fixed point, the intersection of $\ell$ with the $xy$ plane. The line $\ell^*$ is spacelike, null or timelike if and only if $\ell$ intersects $\cC$ at $2,1$ or $0$ points, respectively. These intersections points project to the intersection points of the orbits of the pencil.  Thus the  orbits of  a timelike pencil are disjoint. See Figure \ref{fig:pencils}. 
\end{proof}

\begin{figure}[!htb]    
\centering\includegraphics[width=\textwidth]{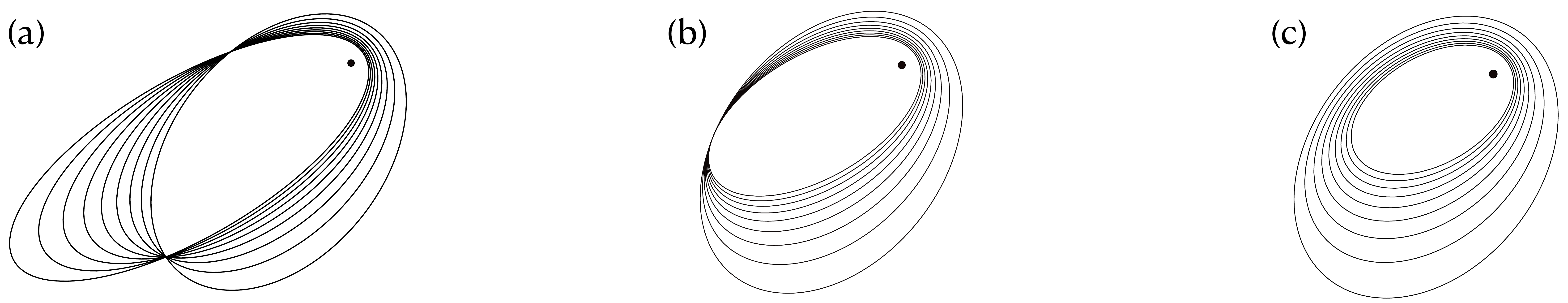}
 \caption{Kepler pencils: (a) spacelike, (b) null,  (c) timelike.}\label{fig:pencils}
 \end{figure}

\begin{remark}[Error alert]\label{rmrk:alert}  Strictly speaking,   items \ref{it:nullcone}-\ref{it:spaceline} of Table \ref{tab:pd}, and the  last two propositions with  their  proof,  are incorrect. Can you see why before continuing reading?

The exceptions arise with the hyperbolic orbits. By  our definition, they   only include one branch (the `attractive branch', see Figure \ref{fig:kepler_orbits}). For example,  there are spacelike pencils of Kepler hyperbolas which only intersect at one point (the 2nd point of intersection is on the `repelling  branch') 
or even spacelike pencils of disjoint Kepler hyperbolas (the 2 intersection points are on the repelling branch). The same problem occurs  with null lines: there are null pencils  of disjoint Kepler hyperbolas (the tangency point is again on the repelling  branch). The proof of Proposition \ref{prop:null} is not correct because applying an orbital symmetry to the circular case may move the tangency  point to a repelling branch. 

Another problem is that  some of the statements are true only when considered in the projective plane. For example, the null line $a=c, b=0$ corresponds to all Kepler parabolas symmetric about the $x$-axis. Their common tangency point lies on  the line at infinity. 

To fix these problems  one needs to separate  statements and proofs of   some items    of Table \ref{tab:pd} into cases.  It is not difficult, and can be even quite entertaining, but we shall not elaborate further on this issue, trusting the reader to make adjustments of the relevant items in the table accordingly. 
\end{remark}

\begin{corollary}\label{cor:fixedMin}
Each   family of Kepler orbits   of fixed minor axis, ellipses or hyperbolas, is a non-flat 2-parameter family admitting a 3-dimensional group of symmetries. The elliptic and hyperbolic cases are not orbitally  equivalent, although in both cases the orbital symmetry group is isomorphic to $\PSLt$. 
\end{corollary}

\begin{proof}The dual surface of such a family is a hyperboloid of  either 1 or 2 sheets, the  `hypersphere'  $a^2+b^2-c^2=\pm 4/B^2$  (items \ref{ti:me}-\ref{ti:mh} of Table \ref{tab:pd}). These  are the level surfaces of the Minkowski norm and are thus invariant under the Lorentz group $\mathrm{O}_{2,1}$, a 3-dimensional subgroup of the full 7-dimensional group of orbital symmetries. This shows that every such family admits at least a 3-dimensional group of symmetries. To show that the family is non-flat, and hence its symmetry group is at most 3-dimensional, we turn to  the same argument in the proof of Theorem \ref{thm:hill1}, explained in the Appendix (Proposition \ref{prop:flatE}). 

Note also that in the elliptic case    the said surface (a spacelike hypersphere) is a translation of the surface corresponding to Kepler orbits of fixed non-zero energy (items \ref{ti:ee}-\ref{ti:eh}). Since translations are generated by orbital symmetries (Theorem \ref{thm:main2}),  the non-flatness follows from Theorem \ref{thm:hill1}. 

The elliptic and hyperbolic cases are not orbitally equivalent, even locally,  because the two actions of the symmetry group $\PSLt$  are  non-equivalent: in the elliptic case the isotropy is an elliptic subgroup  and in the hyperbolic case it is a hyperbolic subgroup, which are non conjugate 1-parameter subgroups of $\PSLt$. 
\end{proof}

\paragraph{The `curved' Kepler problem}(item \ref{ti:cu} of Table \ref{tab:pd}).  There is an analogue of  the Kepler problem on surfaces  of constant curvature $k\neq 0$ (a sphere in $\R^3$ for $k>0$ and a spacelike `hypersphere' in $\Rto$ for $k<0$).  They  are characterized by the property that their {\em unparametrized} orbits centrally project to planar Kepler orbits. See \cite{AlbPr} for more details, where  the following proposition is proved.  

\begin{proposition}\label{prop:crvd}
Central projection maps  orbits of the `curved' Kepler problem on a surface  of constant curvature $k\neq 0$ to  Kepler orbits in $\R^2$. The energy $E_k$ of an orbit in the curved space is related to the energy $E$ of its centrally projected  orbit by
\[E_k = E + \frac{k}{2}M^2,\]
where $M$ is their common angular momentum value.  
\end{proposition}

\begin{corollary}\label{cor:k}
Central projections of   Kepler orbits with energy  $\pm E_k$ on a surface of constant curvature $k$ are  parametrized by the surface $\{a^2 + b^2 -(c - |E_k|)^2=- E_k^2 - k \}\subset \Rto$, where $c>0$ represent orbits with negative energy $E_k = -|E_k|$ and $c<0$  orbits of positive energies, $E_k = |E_k|$. They are the projections to the $xy$-plane of sections of $\cC$ by planes tangent to the surface $(E_k^2 + k)(x^2 + y^2) =kz^2 + 2|E_k|z - 1$ in $\R^3$.
\end{corollary}

The proof is immediate from the last proposition and formulas \eqref{eq:EM}. Let us remark  also that Corollary \ref{cor:k} gives a pleasant dynamical interpretation of Kepler orbits of fixed minor axis: they are the central projections of zero energy orbits of an appropriate curved Kepler problem.

\subsection{A Keplerian version of the Tait-Kneser and 4 vertex theorems} 
\paragraph{Point-line duality.} The equation $ax+by=1$  defines a duality between the $xy$ and $ab$-planes. Namely, each point $(a,b)$ defines a line in the $xy$ plane and vice versa. Given a curve $C$ in one of these planes, its dual $C^*$ is a curve in the other plane, whose points correspond to the lines tangent to $C$. For example, the dual of the circle $x^2+y^2=R^2$ is the circle $a^2 + b^2=1/R^2.$ If $C$ is a smooth strictly convex curve, containing the origin in its interior,  so is $C^*$ and $C^{**}=C.$ This still works if $C$ does not contain the origin in its interior, provided we allow for curves in the projective plane, as we do in the sequel. The tangents to $C$ through  the origin then correspond to intersections of  $C^*$ with the `line at infinity'. 

\begin{proposition}\label{prop:dualC} $C$ is a Kepler orbit  if and only if  $C^*$ is a circle. If $C$ is an ellipse then  $C^*$ contains the origin, if it is a parabola then $C^*$ passes through the origin and if $C$ is an  hyperbola then the origin lies outside $C^*$. In the latter case, the two tangents to $C^*$ through the origin divide $C^*$ into two arcs, corresponding to the two branches of $C$. The larger arc corresponds to the `attractive branch' of $C$ and the shorter to the `repelling branch'. See Figure \ref{fig:duality}.
\end{proposition}

\begin{figure}[!htb]   
\centering\includegraphics[width=.9\textwidth]{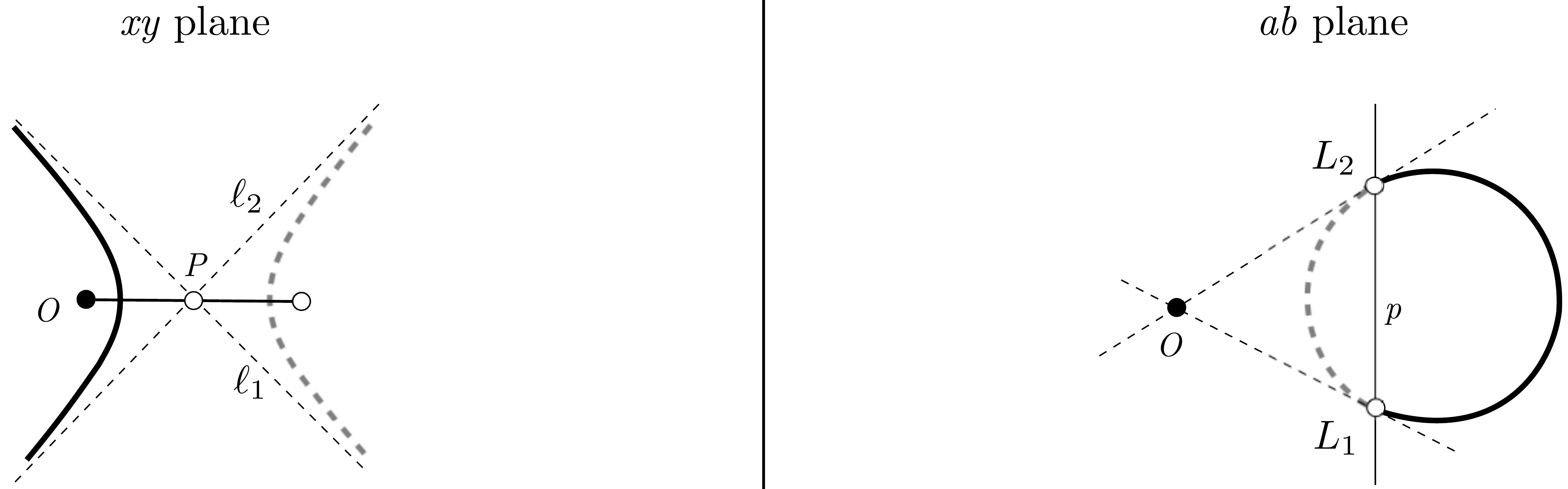}
 \caption{\small Proposition \ref{prop:dualC}}. \label{fig:duality} 
\end{figure}

\begin{proof}Let $\v=(a,b,c)\in\Rto_+$ be the point corresponding  to $C$. The intersection of the null cone through $\v$ with the $ab$ plane is a circle of radius $c$ centered at $(a,b)$. See Figure \ref{fig:null_pencil} (right). The    points of this circle correspond to the lines tangent to $C$ (a special case of Proposition \ref{prop:null}), so the circle is $C^*$. For a parabola, one of its tangents is the line at infinity, whose dual is the origin of the $ab$ plane.

When $C$ is a hyperbola it has  two tangents, its asymptotes, whose tangency points with $C$ are two points on the line at infinity of the $xy$ plane. The two asymptotes correspond to two points on $C^*$ and their intersection points with the line at infinity correspond to the two tangents to $C^*$ at these points, passing through the origin of the $ab$ plane. The longer arc of $C^*$ corresponds to the attractive  branch of $C$  because the latter is nearer the origin then the repelling branch. \end{proof}

\begin{remark} The same warning as in Remark \ref{rmrk:alert} applies here, although it is simpler to fix: if $C$ is a Kepler hyperbola then $C^*$ is not a circle, but rather a circular arc,  corresponding to the Kepler branch of the `full' hyperbola, as shown in Figure \ref{fig:duality}. The complementary arc of the circle corresponds to the 'repelling branch'. 
\end{remark}

\paragraph{Osculating circles.} A plane curve with non-vanishing curvature admits at each of its points an {\em osculating circle}, tangent to the
curve at this point to 2nd order (its curvature coincides with that of the curve at this point). Sometimes the osculating circle is {\em hyperosculating}, i.e. tangent  to order  higher than two. This occurs at the critical points of the curvature and such points are called {\em vertices}. For example, a (non-circular) ellipse has 4 vertices, corresponding to two minima and two maxima of the curvature. The 4-vertex theorem states that {\em on any convex simple planar closed curve there are  at least 4 vertices.}  A related theorem is the Tait-Kneser theorem,  stating that {\em along any vertex-free curve segment with non-vanishing curvature the osculating circles are pairwise disjoint and nested.} Both theorems are  over 100 years old and there are many variations \cite{DGPV, GTT}.

Using Proposition \ref{prop:dualC} above, we shall obtain a Keplerian version of these theorems. To this end, we consider a strictly convex  star-shaped closed curve $\gamma$, that is $\gamma, \gamma'$ and $\gamma', \gamma''$ are everywhere linearly independent (these are parametrization independent conditions). These conditions imply that one can define at each point along $\gamma$ its osculating Kepler orbit, tangent to the curve to 2nd order. A point where the osculating Kepler orbit is hyperosculating is a {\em Kepler vertex}. 

\begin{theorem}\label{thm:TK} There are at least 4 Kepler vertices along $\gamma$. Along any vertex free segment of $\gamma$ the osculating Kepler orbits are pairwise disjoint and nested. See Figure \ref{fig:TK}
\end{theorem}

The proof reduces to the observation that point-line duality preserves order of contact between curves, hence, by Proposition \ref{prop:dualC}, it maps  the osculating Kepler orbit of $\gamma$  to the osculating circle of $\gamma^*$, and the same for hyperosculating Kepler orbits, so it maps Euclidean vertices to Kepler  vertices. It also maps nested Kepler orbits to nested circles, so the theorem is reduced to the Euclidean version. % explain why Kepler branch appears for such convex curves
In a recent article we gave a different proof of this theorem \cite{BT}. 

\begin{figure}[!htb]    
\centering\includegraphics[width=.8\textwidth]{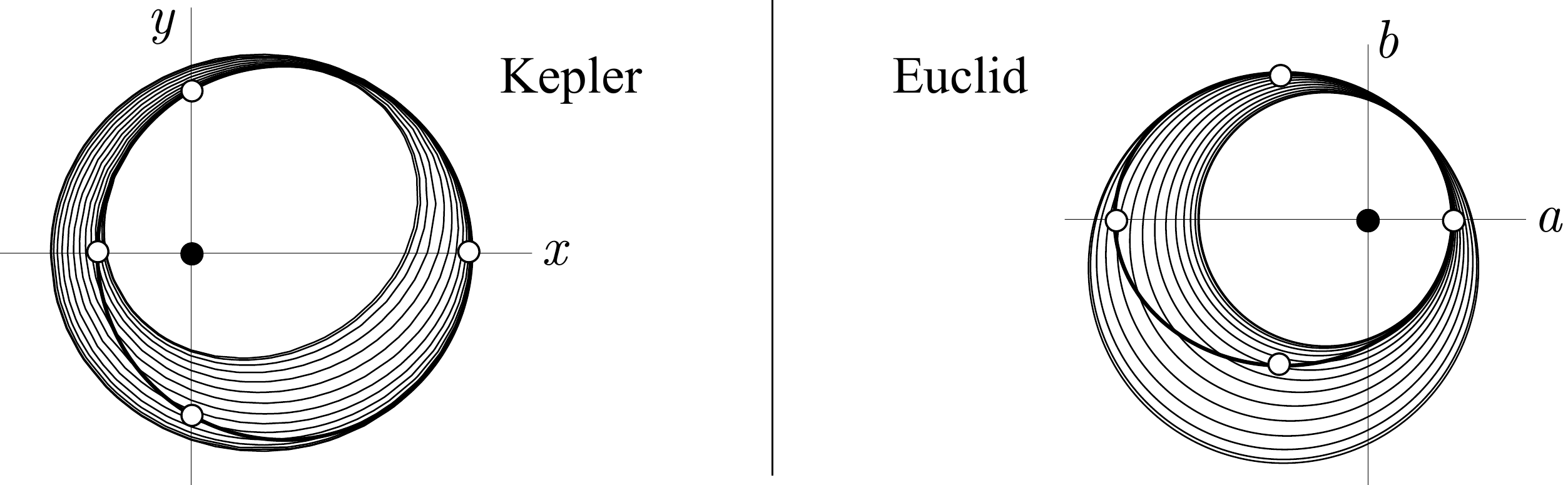}
 \caption{Kepler-Euclid duality. Left: a curve is drawn in the Kepler plane (a  circle centered on the $x$-axis) and the nested family of osculating Kepler orbits along its arc in the 1st quadrant, between 2 of its 4 Kepler vertices (the white dots, intersections of the circle with the coordinate axes).  Right: the dual of the left circle is a Kepler ellipse, with  4 euclidean vertices and osculating circles between two of them, duals of the osculating ellipses on the left.}
 \label{fig:TK}
\end{figure}

\subsection{A minor axis variant of Lambert's Theorem}

Lambert's Theorem (1761)  is a statement about  the  elapsed time along a Keplerian arc \cite{AlbLamb, Se}. Let us recall this theorem. 
Consider  a time parametrized Kepler ellipse, i.e. a solution   $\r(t)$   of  $\ddot\r=-\r/r^3$, with major axis $A$. We fix   two moments    $t_1< t_2$, the corresponding   points  $\r_1=\r(t_1),\r_2=\r(t_2)$, the chord distance $r_{12}=\|\r_1-\r_2\|$,    the distances to the origin  $r_i=\|\r_i\|$   and the   time lapse  $\Delta t=t_2-t_1$. See Figure \ref{fig:lamb}(a). 

\begin{figure}[!htb]    
\centering\includegraphics[width=\textwidth]{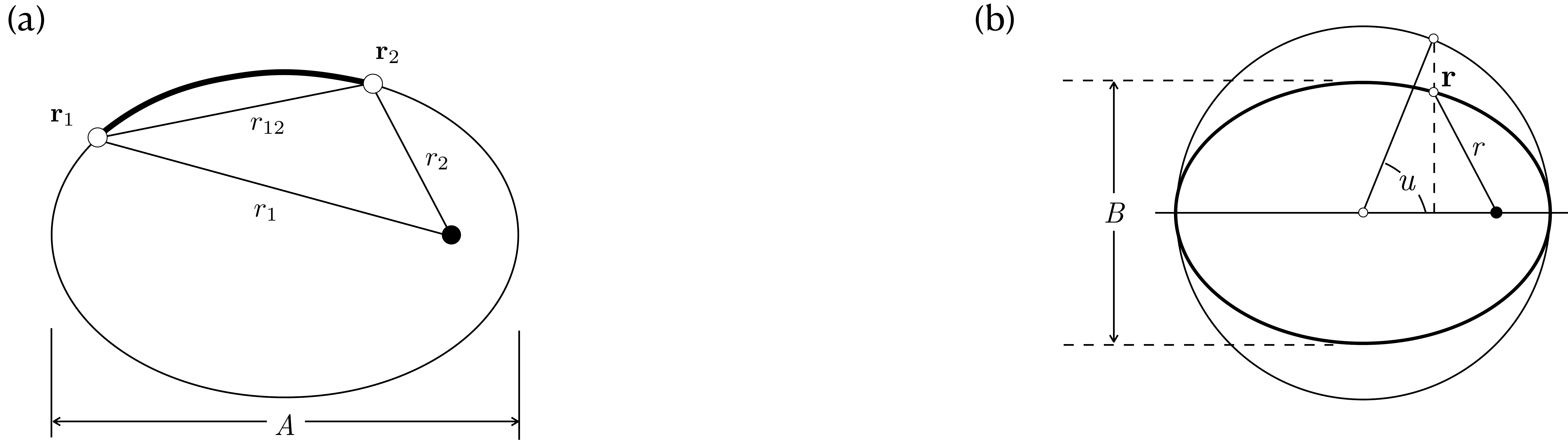}
 \caption{\small (a)  Lambert's Theorem. (b) The  eccentric anomaly $u$.  }
 \label{fig:lamb}
\end{figure}

\paragraph{Lambert's theorem.} {\em $\Delta t$ is a function of $r_{12}, r_1 + r_2$ and $A$.}

\mn  

Clearly, for elliptical orbits the said function is only  well defined modulo the period of the orbit (a function of $A$). The main thrust  of the theorem is that $\Delta t$ does not depend on the individual values of $r_1, r_2$. Thus one can deform the orbit, keeping the three quantities $r_{12}, r_1 + r_2,A$ fixed, into a linear orbit, for which the time $\Delta t$ is easy to write as  an explicit integral. 

Our  `minor axis variant' of this theorem  involves a different well-known parametrization    of Kepler orbits, by the {\em eccentric anomaly} $u$, see Figure \ref{fig:lamb}(b). For simplicity, we shall only deal with Kepler ellipses,  although the statement and proof can be easily modified for parabolic and hyperbolic orbits. 
Consider a Kepler ellipse with minor axis $B$, two values   $u_1<u_2$,  $\r_1=\r(u_1),\r_2=\r(u_2)$, $r_{12}=\|\r_1-\r_2\|$,     $r_i=\|\r_i\|$  and    $\Delta u=u_2-u_1$.

\begin{theorem}\label{thm:Lambert}
 $\Delta u$ is a function of $r_{12}, r_1 - r_2$ and $B$, well defined modulo  $2\pi$. Explicitly, 
\be\label{eq:lamb}
 B^2 \sin^2\frac{\Delta u}{2} = r_{12}^2 - (r_1 - r_2)^2.
\ee
\end{theorem}

\newcommand{\cE}{\mathcal{E}}
\newcommand{\SOto}{\mathrm{SO}_{2,1}}

\begin{proof} We consider an ellipse $\cE$ with minor axis $B$, parametrized by $u$, as in Figure \ref{fig:lamb}(b). We  lift $\cE$ to $\tilde\cE\subset \cC_+$ and $\r_i$ to  $\tilde\r_i=(\r_i, r_i)\in\tilde\cE.$ The right-hand side of Equation \eqref{eq:lamb} is
then $\|\tilde\r_1-\tilde\r_2\|^2$ (using Minkowski's norm),  hence  is invariant under the Lorentz group $\mathrm{O}_{2,1}$. We claim that the left hand  is  invariant as  well, hence it is enough to check formula \eqref{eq:lamb} in the circular case, for which it is immediate. 

To establish the said invariance, we first note that $B$ is $\mathrm{O}_{2,1}$-invariant by  item \ref{ti:me} of Table \ref{tab:pd}. The invariance of $\Delta u$ follows from the next lemma. 
\begin{lemma}\ \\
\vspace{-.5cm}
\begin{enumerate}
\item   Restricted to $\cC$, $dx^2+dy^2-dz^2=(r d\theta)^2.$
\item  Restricted to   $\cE$,   $rd\theta=(B/2)du.$
\end{enumerate} 
\end{lemma}

\mn{\em Proof.} The 1st statement is a simple calculation, using   $x=r\cos\theta, y=r\sin \theta$ and $x^2+y^2=z^2$. For the 2nd statement, from Figure \ref{fig:lamb} we have 
$x  = a ( \cos u - e), 
y =  b \sin u,  r = a ( 1 - e \cos u )$, where $a,b$ are the major and minor semi axes of $\cE$ (respectively) and $e=\sqrt{a^2-b^2}/a$ the eccentricity.  From the  first two equations  follows $dx^2+dy^2=(a^2(\sin u)^2+b^2(\cos u)^2)du^2$ and from the last follows   $dx^2+dy^2=dr^2+r^2d\theta^2=
a^2e^2(\sin u)^2du^2+r^2d\theta^2.$ Equating these two expressions for $dx^2+dy^2$ we obtain $b^2du^2=r^2d\theta^2$, as needed. This  completes the proof of the lemma and also the theorem. 
\end{proof}

\begin{remark} Formula \eqref{eq:lamb} is an elementary geometric statement about ellipses, 
so one expects to find an elementary proof. Indeed, we sketch such a proof here and invite the reader to 
compare it with our proof above. 
Let $a=A/2$, $b=B/2$  (the major and minor semi-axes), $e=\sqrt{a^2-b^2}/a$ (the eccentricity).
Then $  r_j = a( 1 - e\cos u_j) $ and  $r_{12}^2 = a^2(\cos u_1 - \cos u_2)^2 + 
b^2(\sin u_1 - \sin u_2)^2, $ from which follows $r_{12}^2 - ( r_1 - r_2)^2  = 
b^2 \left[  ( \cos u_1 - \cos u_2)^2 + ( \sin u_1 - \sin  u_2 )^2 \right]
 = 
 %4b^2 \left( \sin^2 u_+ \sin^2u_ ++ \cos^2 u_+ \sin^2u_- \right) = 
 B^2 \sin^2 (\Delta u/2).$
\end{remark}

\subsection{Kepler fireworks}
The following intriguing result  is well known.

\begin{proposition}\label{prop:fireworks}Consider the family of Kepler ellipses of fixed (negative) energy, passing through a fixed point. Then there exists a  Kepler ellipse, with second focus at the fixed point,   tangent to all ellipses of the family (the `envelope' of the family). See Figure \ref{fig:fw}(c). 
\end{proposition}
There are many proofs available. For example, Richard's proof \cite[page 839]{R},  using only  elementary Euclidean geometric,  is  hard to beat for simplicity and elegance. We shall prove it following a longer  path,  but will obtain  on the way   two variations on this result, seemingly  new. Let us begin. 

\begin{proposition}\label{prop:hfw}Consider the family of Hooke  (or central)  ellipses of fixed area passing through a fixed point in $\R^2\setminus 0.$ Then these ellipses are all tangent to a pair of parallel lines, symmetric about the line passing through the origin and  the fixed point. See Figure \ref{fig:fw}(a). \end{proposition}

\begin{proof} Without loss of generality, let the  fixed area be $\Delta$ and the fixed point $(1,0)$ (using  rotations and dilations about the origin). Any ellipse of area $\Delta$ passing through $(1,0)$ can be brought by a `shear' $S:(X,Y)\mapsto (X+sY, Y)$ to an ellipse of the form $X^2+(\pi Y/\Delta)^2=1,$ which is clearly tangent to the two lines $Y=\pm \Delta/\pi$. Since $S$ preserves these lines the original ellipse is also tangent to these lines. \end{proof}

This is our  1st variation on Proposition \ref{prop:fireworks} (a rather modest one, admittedly). Before stating  the next variation we use another lemma, possibly of some independent interest. 

\begin{lemma}\label{lemma:bohm} The squaring map $\C\to\C$, $\z\mapsto \z^2$, takes     Hooke ellipses of fixed area to  Kepler ellipses of fixed minor axis. 
\end{lemma}
\begin{proof}Let a Hooke ellipse be $(x/a)^2+(y/b)^2=1$ (without loss of generality). Its area is $\Delta=\pi ab$ and it is parametrized by $X=a\cos \theta, Y=b\sin\theta.$ Its square  is parametrized by $x=X^2-Y^2=(a^2-b^2)/2+(a^2+b^2)\cos2\theta, y=2XY=ab\sin2\theta.$ This is a Kepler ellipse with minor axis $2ab=2\Delta/\pi.$
\end{proof}

Now for   the 2nd variation.

\begin{proposition}\label{prop:minfireworks}Consider the family of Kepler ellipses with  fixed  minor axis and passing through a fixed point in $\R^2\setminus 0$. Then there exists a Kepler parabola tangent to all ellipses of the family (the `envelope' of the family). See Figure \ref{fig:fw}(b). 
\end{proposition}

\begin{proof} By  Lemma \ref{lemma:bohm}, the family of Kepler ellipses with  fixed  minor axis, passing through a fixed point,  is the image under the
 squaring map of the  family of Hooke ellipses of fixed area passing through a fixed point. By Proposition  \ref{prop:hfw}, the envelope of these Hooke ellipses  is a pair of parallel lines, equidistant  from  the   origin. Under the squaring map, the image of these lines is the envelope of the family of Kepler ellipses. Following  this recipe  for the  envelope of the Kepler ellipses with  minor axis $B$ going through  $(x_1,0)$  we get the Kepler parabola $y^2=4p(x+p)$, where $p=B^2/(4x_1).$\end{proof}

\begin{remark} The last proposition can be also  established by passing to the dual statement using Table \ref{tab:pd}, by considering  the parabolic plane in $\Rto$ corresponding to  the fixed point, then taking its polar with respect to the quadric  corresponding to ellipses with a fixed minor axis (hyperboloid of 2 sheets). We leave the details of  this alternate proof   for the reader to explore.
\end{remark}
Now we use duality (Table \ref{tab:pd})  and translation  symmetries  in $\Rto$  (Theorem \ref{thm:main2}) to derive  Proposition \ref{prop:fireworks} from its minor axis variant (Proposition \ref{prop:minfireworks}). 

\mn{\em Proof of Proposition \ref{prop:fireworks}.}   Kepler ellipses with energy $E<0$  passing through 
$(x_0,0)$ correspond to the intersection of  $a^2+b^2-(c+E)^2=-E^2$ with  $ x_0(a+c)=1.$ This is mapped by
$(a,b,c)\mapsto (a,b,c+E)$ to the intersection of $a^2+b^2-c^2=-E^2$ with $x_0(a+c-E)=1.$ The latter are Kepler ellipses with minor axis $B=-2/E$ passing through $(x_1, 0)$, where $x_1=x_0/(1+Ex_0)$, with  envelope   $y^2=4p(x+p)$, where $p=B^2/(4x_1)=(1+Ex_0)/(x_0E^2),$ corresponding to $(-1/(2p), 0, 1/(2p))\in\Rto$. Translating back, the envelope of the original family is given by $(-1/(2p), 0, 1/(2p)-E)\in\Rto$. Using  the value of $p$ and a bit of algebra, this is seen to correspond to  a Kepler  ellipse with 2nd focus $(x_0, 0)$, as needed. \qed

\begin{figure}[!htb]    
\centering\includegraphics[width=\textwidth]{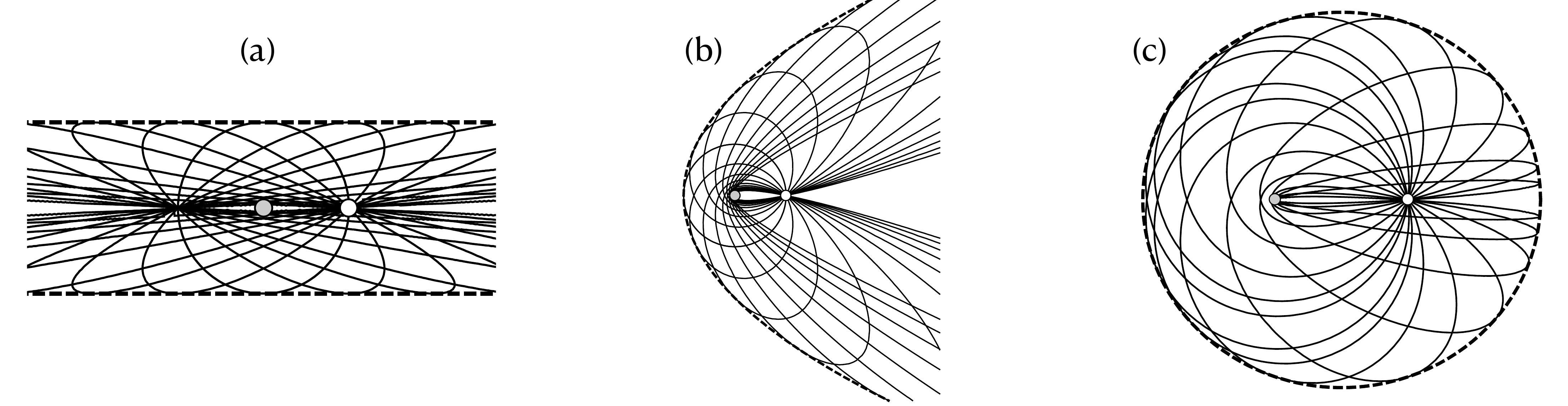}
 \caption {Envelopes of concurrent conics. (a) Hooke's orbits with fixed  area. (b) Kepler's orbits with fixed minor axis. (c) Kepler's orbits with fixed major axis. }
 \label{fig:fw}
\end{figure}

\begin{remark} The  positive energy  analog of  Proposition \ref{prop:fireworks}, i.e. for hyperbolic orbits, is somewhat disappointing, as the family admits no envelope. There is however a `scattering' version of this proposition, for the repelling inverse square law, see Figure \ref{fig:fw_extras}(i). 
A familiar `everyday' version, for constant force, where all orbits as well as the envelope are parabolas, can be observed in fireworks displays and water fountains. See Figure \ref{fig:fw_extras}(b) and (c).   
\end{remark}

\begin{figure}[!htb]    
\centering\includegraphics[width=\textwidth]{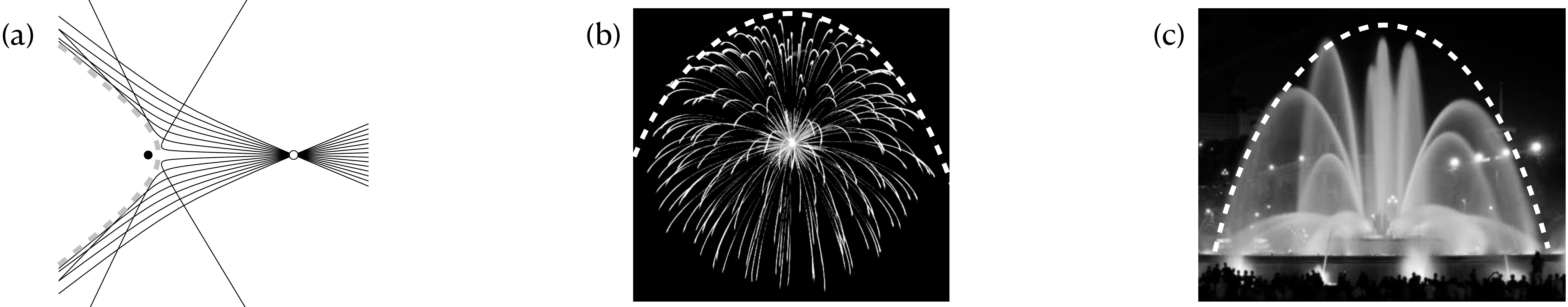}
 \caption{\small (a) Coulomb scattering. (b) Fireworks envelope. (c) Water fountain envelope.}
 \label{fig:fw_extras}
\end{figure}

\section{Proofs of Theorems \ref{thm:main1}-\ref{thm:hill2}}\label{sec:proofs}

\paragraph{Proof of Theorem  \ref{thm:main1}.} 
Let $\RPt$ be the 3-dimensional projective space with homogeneous coordinates $(X:Y:Z:W)$. We identify $\R^3$ with the affine chart $W\neq 0$, $(x,y,z)\mapsto (x:y:z:1).$ 
The closure of $\cC=\{x^2+y^2=z^2\}$ in $\RPt$ is   $\overline\cC=\{   X^2+Y^2=Z^2\}$, obtained by adding to $\cC$ the `circle at infinity' $S^1_\infty=\{ X^2+Y^2=Z^2, \ W=0\}=\ccC\setminus \cC$. See Figure \ref{fig:cylinders}. 

 Let $\tOS\subset \GL_4(\R)$ be the subgroup preserving the (degenerate) quadratic form $X^2+Y^2-Z^2$, 
 up to scale. Its  image $\OS:=\tOS/\R^*$ in the projective group  $\PGL_4(\R)=\GL_4(\R)/\R^*$   is the group  of projective transformations of $\RPt$ preserving $\ccC$. 
\begin{lemma}\label{lemma:g}
 $ \tOS$ consists of elements of the form 
\be\label{eq:g}
\left(
\begin{array}{cc}
A & 0 \\
\b^t & \lambda \\
\end{array}
\right), \quad A\in \COto, \ \b\in\R^3,\ \lambda\in \R\setminus 0.
\ee
\end{lemma} 
\begin{proof}  $g\in\tOS$ if and only if  $g^tJg=cJ$, where  $J=\diag(1,1,-1,0)$ and  $c\in \R$. By a simple calculation $g$ has the claimed form. 
 \end{proof}
 It follows  that $\tOS$ is an  8-dimensional group and  $\OS=\tOS/\R^*$ is 7-dimensional. 
In  the affine chart $\R^3\subset\RP^3$ (column vectors), $\q\mapsto (\q:1)$,  the action  of an element of $\tOS$ given by Equation \eqref{eq:g}  is 
\be\label{eq:action}
\q\mapsto \frac{A\q}{\lambda + \b^t \q}, \quad  \q\in\R^3.
\ee
It   restricts to a local action on $\cC_+$
 and  projects to a local action  on $\R^2\setminus 0$. By the general theory of point symmetries of ODEs (see the Appendix), the maximal  dimension of the symmetry group of a 3-parameter family of plane curves  is 7, hence this local $\OS$-action on $\R^2\setminus 0$ provides  the full group of orbital symmetries.

The expressions for the infinitesimal symmetries  in Equation \eqref{eq:vf} follow from the above by differentiating the action along 1-parameter subgroups of $\tOS$. Let $X\in \mathrm{Lie}(\tOS)$ (the Lie algebra of $\tOS$). Since we are considering projectivized action, we can assume without loss of generality  that $\tr(X)=0$. From  Equation \eqref{eq:g} follows that such an $X$ has the form 

\be\label{eq:X}
X=
\left(
\begin{array}{cccc}
 \frac{x_1}{4} & -x_2 & x_3 & 0 \\
 x_2 & \frac{x_1}{4} & x_4& 0 \\
 x_3 & x_4 & \frac{x_1}{4} & 0 \\
x_5 &x_6 &x_7 & -\frac{3 x_1}{4} \\
\end{array}
\right), \ x_1, \ldots, x_7\in\R.
\ee
 The induced vector field on $\R^2\setminus 0$ is 
$(x,y)\mapsto \gamma'(0),$ where $\gamma(t)=\pi(e^{tX}q)$, $q=(x,y,\sqrt{x^2+y^2},1)^t$ and $\pi(X,Y,Z,W)=\left(X/W, Y/W\right).$ The  formulas of Equation \eqref{eq:vf}  follow from  this recipe by  setting  $x_i=1$ and the rest 0 in Equation   \eqref{eq:X}, $i=1, \ldots, 7.$  \qed

\paragraph{Proof of Theorem \ref{thm:main2}.}
 Note first that an element $g\in\tOS$, given by Equation \eqref{eq:g}, acts on $(\R^4)^*$ (row vectors) by $p\mapsto pg^{-1}$. In the affine chart  $\Rto\subset\mathrm{P}((\R^4)^*)$ (row vectors), $\p\mapsto (\p:-1)$, the action on $\Rto$ by an element of $\tOS$,  given by Equation \eqref{eq:g}, is 
\be\label{eq:dualaction}
\p\mapsto (\lambda\p +\b^t) A^{-1}, \  \p\in\Rto. 
\ee
It follows that for  $X$ given by Equation \eqref{eq:X} the  induced  vector field on $\Rto$  is $\p\mapsto \gamma'(0),$ where $\gamma(t)=\pi(p e^{-tX})$,  $p=(\p,-1)$ and $\pi(A,B,C,D)=-\left(A/D, B/D, C/D\right).$ \qed

\paragraph{Proof of Theorem \ref{thm:parab}.} Identify $\R^2=\C$ and consider the squaring map $B: \z\mapsto \z^2.$ 

\begin{lemma} $B$ defines  a $2:1$ cover $\C\setminus 0\to\C\setminus 0$, mapping pairs of parallel symmetric affine lines   into  Kepler parabolas. 
\end{lemma}
\begin{proof} Since   $B$ is $\C^*$-equivariant, $B(\lambda Z)=\lambda^2B(Z),$ $\lambda\in\C^*$, it is enough to consider the pair  $x=\pm 1$. Their  $B$-image is the Kepler parabola $x=(1+y/2)^2.$ 
\end{proof}
It follows that the set of Kepler parabolas is a flat 2-parameter family of plane curves. \qed

\paragraph{Proof of Theorem \ref{thm:angmom}.} We offer two proofs.

\mn {\em First proof.} Kepler orbits with angular momentum $M$ are the projections of sections of $\cC$ by planes passing through $P:=(0,0,M^2)$ (Corollary \ref{cor:M}). Central projection from $P$ then maps these conic sections  to straight lines in the $xy$ plane. %Thus the family is  locally equivalent to the  2-parameter family of lines in $\R^2$.

\mn {\em Second proof.} Kepler orbits with fixed $M$ are parametrized by the horizontal plane $\{c=1/M^2\}\subset \Rot$, see Corollary \ref{cor:M} above. We know that $\OS$ acts on $\Rto$ as its full group of Minkowski similarities, so there is an element  $g\in\tOS$ that translates this plane to the  plane $c=0$, parametrizing straight lines in the $xy$ plane. By  Equation \eqref{eq:dualaction}, we can take $g$ corresponding  to $A = id, \b = (0,0,-1/M^2)$.
The stated formula  follows  from Equation \eqref{eq:action}.\qed

\begin{remark} Yet another proof, less elementary,  is to write a second order linear ODE for the family of Kepler orbits with fixed $M$ and use the fact that  second order linear ODEs are  flat \cite[page 44]{Ar2}. The said ODE is $\rho''(\theta)+\rho(\theta)=1/M^2,$ where $\rho=1/r$. See the proof of Proposition \ref{prop:flatE} below. 

\end{remark}

\paragraph{Proof of Theorem \ref{thm:hill1}.} 
According to the general theory of  symmetries of ODEs, flatness of a 2-parameter family of plane curves  is  equivalent to the  vanishing of certain two differential  invariants  of  an associated second order ODE. In the Appendix we carry out a calculation showing that one of these invariants is non-vanishing for the family of Kepler orbits of fixed non-zero energy, thus proving  that each such family  is non-flat, see Proposition \ref{prop:flatE}.  Next, according to another  basic result of the  theory, the  dimension of the symmetry group of a non-flat 2-parameter family is at most 3. Thus, for each $E\neq 0$, it is enough to  find  a 3-dimensional subgroup of $\OS$ preserving the set of Kepler orbits with energy $E$. 

As explained in Corollary \ref{cor:E}, Kepler orbits with energy $\pm E\neq 0$ are projections of sections of $\cC$ by planes tangent to the inscribed paraboloid of revolution $\cP= 
\{2z=|E|\left(x^2+y^2\right)+1/|E|\}$. Let $\ccP$ be the closure of $\cP$ in $\RP^3$.
It is a smooth convex compact surface, 
 given in homogeneous coordinates by  the vanishing of the quadratic form $|E|\left(X^2+Y^2\right)-2ZW+W^2/|E|,$  obtained by adding to $\cP$ the  point $(0:0:1:0)$, the tangency point of $\ccP$ with the plane $W=0$ (the white dot in  Figure \ref{fig:huevo}(a)). Consider the subgroup  $\tOS_E\subset \tOS$  preserving this quadratic form up to scale. A short calculation shows that its Lie algebra consists of matrices of the form 
\be \label{eq:XX}
X=\left(
\begin{array}{cccc}
 0 & -x_2 & x_3 & 0 \\
 x_2 & 0 & x_4 & 0 \\
 x_3 & x_4 & 0 & 0 \\
 |E| x_3 & |E| x_4 & 0 & 0 \\
\end{array}
\right), \quad x_2, x_3, x_4\in \R.
\ee
The associated vector field in the $xy$-plane is  $(x,y)\mapsto \gamma'(0)$, where $\gamma(t)=\pi(e^{tX}q)$, $q=(x,y,\pm\sqrt{x^2+y^2},1)
^t$ and $\pi(X,Y,Z,W)=\left(X/W, Y/W\right).$ The sign in $q$ is the opposite sign of $E$, since for $E>0$ (the hyperbolic case) we need to project the action from $\cC_-$ and for $E<0$  from $\cC_+$. Setting  $x_i=1$ and the rest 0 in Equation   \eqref{eq:XX}, $i=2,3,4, $  we obtain from this recipe for $E<0$ the vector fields 
$$v_2:=\partial_\theta,\ v_3:=r(\partial_x+Ex\partial_r),\ v_4:=r(\partial_y+Ey\partial_r),$$
as in  Equation \eqref{eq:fe}. 
For $E>0$ we get the vector fields $v_2,-v_3,-v_4.$ In both cases, $v_2, v_3, v_4$ are infinitesimal generators of the $\OS_E$-action, as stated.

The isomorphism $\tOS_E/\R^*\simeq \PSLt$ is best seen in the dual picture, in $\Rto$. See Figure \ref{fig:huevo}(b). 

\begin{figure}[!htb]
        \centering
        \includegraphics[width=.9\textwidth]{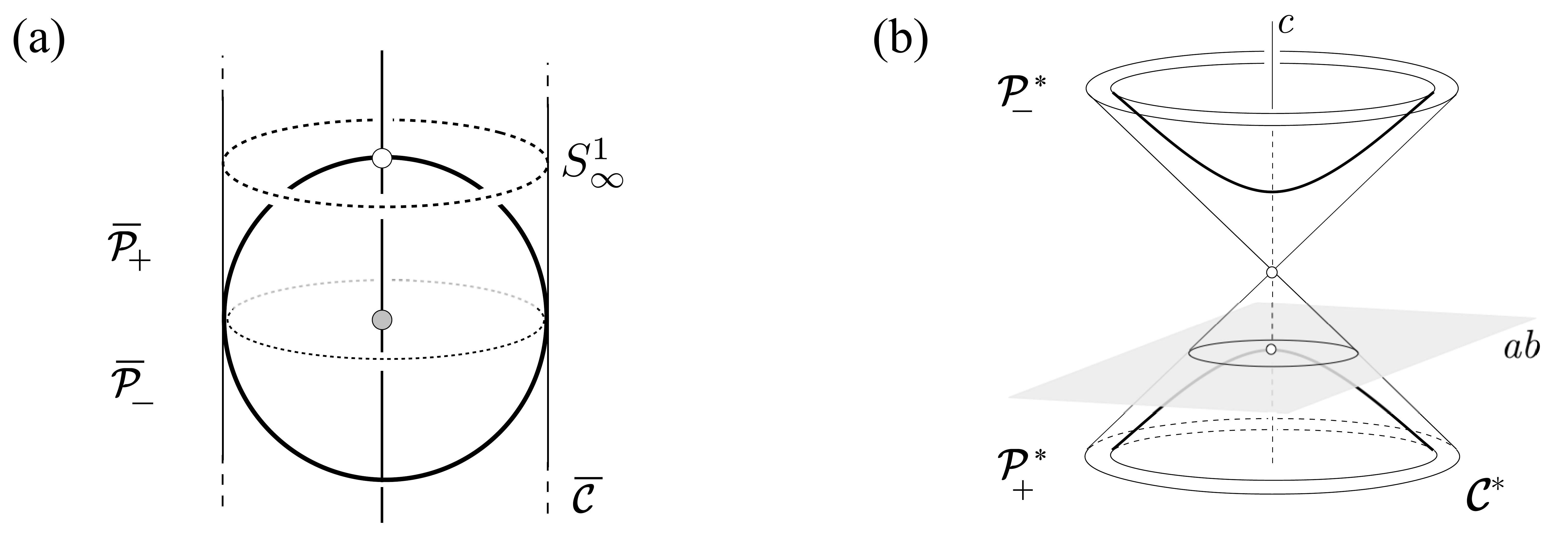} 

        \caption{\small The proof of Theorem \ref{thm:hill1}. (a) In the affine chart $Z\neq 0$, with coordinates $x=X/Z, y=Y/Z, w=W/Z$ the surface  $\ccP$ is the ellipsoid of revolution $x^2+y^2+(w-E)^2/E^2=1,$ inscribed in  the vertical cylinder $\ccC=\{x^2+y^2=1\}$,  tangent to the plane $w=0$ (the `plane at infinity' in the chart $W\neq 0$). Compare to  Figure    \ref{fig:parabo}, where $\cP$ is drawn in the chart $W\neq 0$. (b) The dual picture where $\cP^*$ parametrizes planes tangent to $\cP$. It is a hyperboloid of revolution of two sheets. The Minkowski metric restricts to a hyperbolic metric on it, and $\OS_E$ acts as its group of isometries. 
        }
       \label{fig:huevo}
  \end{figure}

Kepler orbits of energy $E\neq 0$  are parametrized by the surface $\cP^*=\{-a^2-b^2+(c-|E|)^2=E^2\}\subset \Rto$, the quadric surface dual to $\ccP$ (see Equation \eqref{eq:EM} and Figure 
\ref{fig:huevo}(b)). This is a hyperboloid of revolution of two sheets. The lower sheet $\cP^*_+$ parametrizes planes tangent to $\cP_+$, which correspond to Kepler hyperbolas  with energy $|E|$. Similarly for the lower sheet. The Lorentzian metric $da^2+db^2-dc^2$ in $\Rot$ restricts to an hyperbolic metric on each of the sheets, on each of which the identity component of $\OS_E$ acts as the identity component of its isometry group (in the full $\OS_E$ there is also an element interchanging the two sheets, we will use it in the proof of the next theorem). 

It is also clear from Figure \ref{fig:huevo}(a) why the orbital symmetry action on $\cH_E$ for $E>0$ is only local. This is because $\ccP_+$ touches the plane $W=0$ (the `plane at infinity' of the affine chart $W=0$, intersecting $\ccC$ at $S^1_\infty$) at one point, which does not correspond to any point in Kepler's $xy$ plane. \qed

\paragraph{Proof of Theorem \ref{thm:hill2}.} Consider in  Figure \ref{fig:huevo}(b) the reflection about the horizontal plane $c=|E|$ passing through the vertex  of the shown cone, $(a,b,c)\mapsto (a,b, 2|E|-c),$
interchanging  the lower  and upper sheets $\cP^*_\pm$ of $\cP^*$. The corresponding element in $\tOS$ is 
$$g=\left(
\begin{array}{cccc}
 1 & 0& 0 & 0 \\
 0& 1& 0 & 0 \\
 0 & 0 & -1 & 0 \\
 0 & 0 & -2|E| & 1 \\
\end{array}
\right).
$$
In Figure \ref{fig:huevo}(a), in the affine chart $Z\neq 0$ with coordinates $x=X/Z, y=Y/Z, w=W/Z$, $g$ acts by  $(x,y,w)\mapsto (x,y,2|E|-w)$, a reflection about  the center $(0,0,|E|)$ of $\ccP$ (the dark dot), interchanging $\ccP_\pm$. 
In Figure \ref{fig:parabo}, in the affine chart $W\neq 0$, with coordinates $x=X/W, y=Y/W, z=Z/W$, $g$ acts by $(x,y,z)\mapsto (x,y,-z)/(1-2|E|z),$ interchanging $\cP_\pm.$

To write an explicit orbital embedding $\cH_E\to \cH_{-E}$, note first in Figure \ref{fig:parabo} that Kepler hyperbolas are the projections of sections of the {\em lower} part  $\cC_-$ with planes tangent to $\cP_+$, and that Kepler ellipses are the projections of sections of the {\em upper} part $\cC_+$ with planes tangent to $\cP_-$. The embedding is thus given by the composition    $\r=(x,y)\mapsto (\r,-r)\mapsto (\r,r)/(1+2Er)\mapsto \r/(1+2Er),$ as needed. 

We can also map the `repelling branches' of Kepler hyperbolas with energy $E$ into $\cH_{-E}$, but these are the projections  of sections of the {\em upper} part of $\cC$ with planes tangent to $\cP_+$, thus the embedding is $\r=(x,y)\mapsto (\r,r)\mapsto (\r,-r)/(1-2Er)\mapsto \r/(1-2Er).$ See Figure \ref{fig:emb}. \qed

\appendix

\section{Appendix: Symmetries  of  ODEs}
\label{app:ode}

The purpose of this appendix is twofold: first, we  fulfill a promise made in the beginning of the proof of Theorem  \ref{thm:hill1}, showing  that the 2-parameter family of Kepler orbits with fixed non-zero energy is not flat. See Theorem \ref{thm:flat} below. Second, we fit   the results of this article  into the general context of the theory of symmetries of ODEs.

\paragraph{Lie's theory of symmetries of ODEs.}  An $n$-parameter family of plane curves is given, locally,  under some mild regularity  conditions, by the graphs of solutions $y(x)$ of an $n$-th order ODE $y^{(n)}=f(x,y,y',\ldots, y^{(n-1)}).$ Local diffeomorphisms of the $xy$ plane preserving the graphs of solutions of the ODE are classically called {\em point symmetries} of the ODE. Vector fields in the plane whose flow acts by point  symmetries  are {\em infinitesimal point symmetries}. The subject    was developed in the 19th century,  mostly by Sophus Lie and his students, later on  in the 20th century  by \'E.~Cartan and many others,  and  is a still an active area of research. A standard modern reference is P.~Olver's book, see also \cite{BK, DL, PS, S}. 

\paragraph{On `local symmetries'.}
Point symmetries are   local not only in the $xy$  plane but also in the jet spaces over $\R^2$ to which they are naturally prolonged. 
An $n$-th order ODE $y^{(n)}=f(x,y,y',\ldots, y^{(n-1)})$ defines a hypersurface  $M:=\{p_n=f(x,y,p_1, \ldots, p_{n-1})\}$  in the total space $J^n$ of the bundle of $n$-th order jets of curves in $\R^2$. $M$ is an $(n+1)$-dimensional manifold, doubly foliated,  with leaves of dimensions $n-1$, $1$, the sum of  whose tangents  span a contact distribution on $M$. The  first  foliation is  by the fibers of the projection $(x,y,p_1,\ldots, p_n)\mapsto (x,y)$ and the second by the $n$-th jets  of the solutions to the ODE. A point symmetry of the  ODE is a local diffeomorphism  of $M$ preserving both foliations. It projects to a local diffeomorphism of the $xy$ plane. A  good introduction to this geometric point of view on ODEs, for $n=2$, is Arnold's book \cite[Section 1.6]{Ar2}.

%, whose natural lift to $J^n$, {\em restricted to some open subset},  preserves the hypersurface. 

%\sn  (b) More generaly, an $n$-parameter family of plane curves is an  $n+1$ dim mnfld $M$, equipped with a pair of integrable distributions $D, D^*$ of coranks 2 and $n$, spanning a contact dist on $M$ (a bracket generating corank 1 dist).  Point symmetries of the family are local diffeos preserving this pair of distributions. Every such structure $(M,D,D^*)$ is locally equivalent to to the one arising from an $n$-th order ODE, where $M$ is hypersurface in $J^n$ defined by the ODE, $D$ is the tangent to the fibers of the projection $J^n\to \R^2$, $(x,y,p_1,\ldots, p_n)\mapsto (x,y)$, restricted to $M$, and $D^*$ the tangent to the fibers of the projection to the space of solutions of the ODE. 

 \paragraph{Flat families.}
  An $n$-parameter family of plane curves is {\em flat} if it is locally diffeomorphic to  the family given by $y^{(n)}=0$ (graphs of polynomial functions  of degree $<n$). As was shown by S. Lie, a family is flat if and only if its local symmetry group  is  $(n+4)$-dimensional for $n>2$ and 8-dimensional for $n=2$, % (the projective group $\PGL_3(\R)$), 
the maximal dimension possible for an $n$-parameter family of plane curves (Theorems 6.39 and 6.42 of \cite{O}). %In particular, all maximally symmetric cases are locally diffeomorphic.  

%\paragraph{Flatness of the Kepler orbits (revisited).}  
The $n=3$ case, i.e. point symmetries of 3rd order ODEs,  was further studied in more depth in  1905 by K. W\"unschmann \cite{W}, around 1940 by S.-s Chern \cite{Ch1, Ch2} and  \'E. Cartan  \cite{Cart},     and later on by others \cite{Go, GN1, GN2, SY, To}.  The only result  from this theory  that we use, in the proof of Theorem \ref{thm:main1},  due to Lie, is  that  {\em the maximum dimension of the symmetry group of a 3-parameter family of plane curves is 7.} 

 Theorem \ref{thm:main1} can thus be interpreted as saying that the 3-parameter family of Kepler orbits is locally diffeomorphic  to the solutions of $y'''=0,$  i.e. vertical parabolas of the form $y=ax^2+bx+c.$ Let us find  such a diffeomorphism. 
 Define a map from the $XY$ plane to the $xy$-plane by 
\be \label{eq:vp}(X,Y)\mapsto (x,y)=\left({X^2- 1\over Y}, {2X\over Y}\right).
\ee

\begin{proposition}
Equation \eqref{eq:vp} defines   a local diffeomorphism from the  $XY$-plane into the $xy$-plane, mapping   each  vertical parabola  $Y=AX^2+BX+C,$ $A,B,C\in\R$, onto the Kepler orbit  
$ax+by+cr=1$, where $a=(A-C)/2, b=B/2, c=(A+C)/2.$ 
\end{proposition}

The proof is by a straightforward verification.

\paragraph{Path geometries, Tresse classification.}
 The $n=2$ case  is the best known  and is called a {\em path geometry.}  If a 2-parameter family is not flat then the maximal possible dimension of the symmetry group drops from 8 to 3. A list of normal forms of 2nd order ODEs admitting a 3-dimensional group of symmetries, over the complex numbers,  was derived by A. Tresse (a French student of S. Lie) in his 1896 PhD dissertation \cite{Tr}.  The list is divided into 4 `types', according to the symmetry group (all types come with 1 or 2 continuous parameters). Type d), the type that concerns us, deals with  $\SL_2(\C)$ invariant 2nd order ODEs, and is  given by Tresse as   $y''=(a(y')^3-y')/(6x),$  where $a$ is a (complex) parameter. 
 
 Tresse classification  was  extended to the real case \cite{DK, La} 
 but by and large we think that this list  has not been sufficiently explored. 
 
  Over the reals, Tresse's type d) breaks first into two subtypes, according to the two real forms of $\SL_2(\C)$: $\SU_2$ and $\SL_2(\R)$. We are concerned with $\SL_2(\R)$. 
  
  Among the $\SLt$-invariant path geometries, there are two `exceptional'  cases (without parameters), corresponding to the two ODEs $y''= \pm (xy' - y)^3$.
  % (equivalent normal forms are  $y''=(\mp (y')^3-y')/(2x)$). 
  What distinguishes these two cases from all other items on Tresse list is that these are the only cases of {\em projective} path geometries, i.e. the paths  are the (unparametrized) geodesics of a torsionless affine connection. In fact, in this case the paths are the geodesics of the  well known Jacobi-Maupertuis  metric defined on the Hill region for any mechanical system with fixed energy.

  The case that appears here (constant energy Kepler orbits) corresponds to  $y''=(xy'-y)^3$, but it is not so easy to see the equivalence (we  will not pursue it here).

%\paragraph{More on path geoemetries, their duals, differential invariants.}
A  path geometry on a surface $S$  determines a `dual' path geometry on the path space $S^*$, parametrized by the points of $S$: to each point of $S$ is assigned a path in $S^*$, the set of paths in $S$ passing through this point. The dual path  geometry of a flat path geometry (straight lines, graphs of solutions to $y''=0$) is also flat, but a generic non-flat path geometry is not equivalent to its dual. The flatness  of a path geometry, given by a 2nd order ODE $y''=f(x,y',y'')$, 
%
%\note{added ``relative invariants''}
%
 is detected by the vanishing of  the relative invariants 
\begin{align}
\begin{split}\label{eq:I}
I_1=&f_{pppp}, \\
I_2=&D^2f_{pp} -4Df_{py}  + f_p(4f_{py}- Df_{pp})- 3f_{pp}f_y + 6f_{yy},
\end{split}
\end{align}
where $p=y'$ and $D=\partial_x+p\partial_y+f\partial_p.$ 

The vanishing of $I_1$  simply means  that $F$ is at most cubic in $y'$. This  is a diffeomorphism invariant property, characterizing  projective path geometries. The vanishing of  $I_2$ is equivalent to the projectivity of the dual path geometry. Thus a path geometry is flat if and only if it is projective and its dual path geometry is projective as well.

\paragraph{Kepler orbits of fixed energy.} We can now fill the  gap left out in the proof of Theorem \ref{thm:hill1}. 
\begin{proposition}\label{prop:flatE}Kepler orbits of fixed energy $E\neq 0$ form  a non-flat path geometry. In fact, $I_1=0$ but $I_2\neq 0$. Thus the maximum dimension of the symmetry group of such a family is 3. 
\end{proposition}
\begin{proof}We 1st write down a 2nd order ODE for Kepler orbits of energy $E$. Using the equation $ax+by+cr=1$  of Theorem \ref{thm:EM}(a), we get 
$$\rho=a \cos\theta+b\sin\theta+c,\ \rho'=-a \sin\theta+b\cos\theta, \ \rho''=-a \cos\theta-b\sin\theta, $$
where  $x=r\cos\theta, y=r\sin\theta, r=1/ \rho.$ It follows that 
$$\rho+\rho''=c,\ (\rho')^2+(\rho'')^2=a^2+b^2.$$
Using this in $ 2cE=a^2+b^2-c^2 $ (Equation \eqref{eq:EM} with $c>0$), we get, 
$$\rho''={\rho^2+\rho'^2\over 2(\rho+E)}-\rho.$$
Using  Equations \eqref{eq:I} we get  $I_2=9E^2/(E+\rho)^3,$ hence $I_2\neq 0$ for $E\neq 0.$
\end{proof}

\begin{remark}Incidentally, the formula  $I_2=9E^2/(E+\rho)^3$ of the last proof gives another proof of Theorem \ref{thm:parab}.
\end{remark}
\paragraph{Central forces with flat  orbit space. The W\"unschman condition.} Theorem \ref{thm:main1} establishes that Kepler orbits form a flat 3-parameter family of curves, i.e. locally diffeomorphic to the family of vertical parabolas, given by $y'''=0$. Using the squaring   map, $\z\mapsto \z^2$,  this result extends to Hooke orbits, the family of central  conics, trajectories of a mass under  Hooke's force laws, $\ddot\r=\pm \r.$ 
{\em Are there any other force laws,  whose orbits form a flat family of plane curves?}

We do not know the answer in general. But for {\em central} force laws, i.e. Newton's equations of the form $\ddot\r=f(r) \r/r,$  the answer is negative. To prove it,  we show that in fact the Hooke and Kepler  laws are the only central force laws satisfying a condition weaker than flatness, called the  {\em W\"unschman condition} (1905). Given a 3-parameter family of plane curves, one defines null cones in the parameter space whose rulings consist of the curves that are tangent to a fixed line at a fixed point. In the flat case, such as the space of Kepler orbits,  these cones are quadratic and thus define a (flat) conformal structure on the parameter space. However, for a general family, these cones may fail to be quadratic. The families for which the null  cones are  quadratic, and hence define a conformal Lorentzian metric on the parameter space, are characterized by a complicated PDE on the ODE  that  defines this family, studied by K. W\"unschmann \cite{W}. For  a modern presentation of this deep result see \cite{N}.

\begin{theorem}\label{thm:flat} The orbits of the system $\ddot\r=f(r) \r/r$ form a flat 3-parameter family of plane curves if and only if $f(r)$ is a constant  multiple of $r$ or $1/r^2$. In fact, these force laws are the only central ones satisfying the W\"unschmann condition. 
\end{theorem}

\begin{proof} Following the standard procedure outlined above, we first write a 3rd order ODE whose solutions are the (unparametrized) orbits of the system $\ddot\r=f(r) \r/r$,

\be\label{eq:cent}\rho'''=\rho' \left[(\rho'' + \rho)\left({f'(\rho)\over f(\rho)} - {2\over \rho}\right) - 1\right],\ee
where $\rho=1/r, $ $\rho=\rho(\theta)$ (see for example \cite{K}). Next, the W\"unschmann condition for $\rho'''=F(\rho, \rho', \rho'')$ is 
$$F_\rho + \left(D  - {2\over 3}F_{\rho''}\right) K=0,$$
where
$$K={1\over 6}DF_{\rho''}-{1\over 9}F_{\rho''}^2-{1\over 2}F_{\rho'}, \quad D=\partial_\theta+\rho'\partial_\rho+\rho''\partial_{\rho'}+F'\partial_{\rho''}.$$
See \cite[Equation 8]{N}. Applying this condition to the right hand side of Equation 
\eqref{eq:cent}, we get a pair of ODEs for $f(\rho)$, whose only solutions are constant multiples of $\rho^2$ and $1/\rho.$
\end{proof}

\paragraph{Central forces   and projective path geometries.}
As mentioned above, in the   local classification of path geometries admitting a 3-dimensional group of symmetries there are only 3 projective cases, where the paths arise as the unparametrized geodesics of a torsionless affine connection. In general, a projective path geometry need not be a {\em metric}  path geometry, i.e. the affine connection may not be  the Levi-Civita connection of  a pseudo-Riemannian metric, but in our  3 cases they are metric connections. In fact, all 3 cases arise as the orbits of fixed energy of conservative mechanical systems, and thus can be realized as geodesics of the associated Jacobi-Maupertuis metric. Let us list the 3 cases by  2nd order ODEs defining them:

\begin{itemize}

\item[I.] $y''=0$. 

\item[II.]  $y''=(xy'-y)^3$

\item[III.]  $y''=-(xy'-y)^3$
\end{itemize}
(See e.g. \cite{DK}, where our type I is item 4 of Theorem 7 and our types II and III are items $3d_+$ and $3d_-$ ,  respectively.) 

Type I is the flat path geometry, admitting an  $8$-dimensional symmetry group, the projective group $\PGL_3(\R)$. 
Type II and III are non-flat, each admitting $\SLt$ as a local symmetry group. In both types II and III the $\SLt$ action is locally equivalent to the standard linear action on $\R^2\setminus 0$. The dual actions, on the dual path geometries,  are non equivalent: for the dual of type II    $\SLt$ acts by isometries of the hyperbolic plane and in the dual of type III as isometries of pseudo-hyperbolic plane (non-flat  constant curvature Lorentzian metric).   Both actions appear naturally as open orbits of the projectivized adjoint representation of $\SLt$. 

In Table \ref{tab:proj} we place some 2-parameter families of curves arising naturally in planar mechanical   systems  with central-force laws,  locally realizing the 3 path geometries. In the 1st two rows we consider central-force power laws,  $\ddot\r=f(r)\r/r,$ $f(r)=\pm r^\alpha,$ where    $M$ and $E$ are the  (fixed) angular momentum and energy, respectively. 
In parentheses is  the force   law ($\pm r^\alpha$, with `--' for attractive and `+' for repelling). In the following two rows $E_k$ is the energy, $M_k$ the angular momentum, for the Kepler problem in a space of constant curvature $k$, as in  \cite{AlbPr}.

\begin{table}[!htbp]%[ht]
\renewcommand{\arraystretch}{2}
\captionsetup{ font=normalsize}
\caption{Projective path geometries and central-force laws}\label{tab:proj}\label{tab:proj}
{
%\footnotesize
%\small
\centering
%\hspace{-.5cm}
 \begin{tabular}[t]{p{.3\textwidth}| p{.3\textwidth}|p{.3\textwidth} } 
I. $y''=0$. 
& II. $y''=(xy'-y)^3$
& III.  $y''=-(xy'-y)^3$\\ 
\cmidrule(lr){1-3}
  $M\ne 0$, ($\pm 1/r^2, \pm 1/r^3$) &  $M\ne 0$, ($-r$)  &  $M\ne 0$, ($r$)  \\ 
$E=0$, ($\pm r^\alpha $, $\alpha\neq -1$) &  $E\ne 0$, ($\pm 1/r^2, \pm r$) & \quad --\\
$|E_k| = \sqrt{-k}$, $k<0$ &  $|E_k|>\sqrt{-k}$, $k<0$ &  $|E_k|<\sqrt{-k}$, $k<0$  \\
 $M_k\ne 0$ &  $E_k$, $k>0$ & \quad--\\  
\end{tabular}
}
\end{table}

\paragraph{Some comments  on Table \ref{tab:proj}.}\ 

\sn{\bf 1.}  `Hooke' orbits, attractive or repelling  ($f=\pm r$),  with fixed  angular momentum $M$, were placed in the table by considering the   squaring  map, $\z\mapsto \z^2$.  They are thus mapped to Kepler orbits with fixed {\em minor} axis.  Attractive Hooke orbits ($f=-r$) are mapped to Kepler ellipses with fixed minor axis (see item \ref{ti:me} of Table \ref{tab:pd} and Lemma \ref{lemma:bohm}), which are equivalent to ellipses of constant energy (see proof of Corollary \ref{cor:fixedMin}), corresponding to type II path geometry. Repelling Hooke orbits   ($f=r$) are mapped to Kepler hyperbolas with fixed minor axis (item \ref{ti:mh} of Table \ref{tab:pd}), which is type III path geometry. 

 \sn{\bf 2.}  Zero energy   orbits for all   central-force power laws,  $f=\pm r^\alpha$, $\alpha\neq -1$, can be seen to give a flat path geometry (type I) by using the Jacobi-Maupertuis metric: by making the change of variable $r=\rho^{2/(\alpha+3)}$ for $\alpha\neq -3$, or $r=e^\rho$ for $\alpha=-3$, one shows that such families  are equivalent to geodesics on a quadratic cone, so are locally equivalent to lines in the plane \cite[\S4]{Mont}. More generally,  for planar motion $\ddot\r=-\nabla U$, with  potential  satisfying $\Delta\log U = \lambda U$ for some $\lambda\in\R$, the orbits at energy zero will also be locally flat.

 \sn{\bf 3.}  By computing the relative invariants $I_1, I_2$ of Equation \eqref{eq:I}, it can be shown that  orbits with fixed non-zero energy    are non-flat for all   central-force power  laws. It also shows that zero energy orbits for $f=\pm r^\alpha$ are flat if and only if $\alpha\neq -1$. 
Furthermore, by using additional  (relative) invariants \cite[\S6]{DK}, one finds that these path geometries admit a 3-dimensional symmetry group only for the  Hooke and Kepler laws  ($\alpha=1, -2$).  
   
% \note{GB: added `relative' in front of `invariant'}
 \sn{\bf 4.} Using  $I_1, I_2$, it can be  also shown that among all central-force power laws, orbits at a fixed non-zero angular momentum are  flat only for the  Kepler and inverse cubic force laws ($\alpha=-2, -3$).

\end{document}